\documentclass[a4paper,reqno]{amsart}

\usepackage[british]{babel}

\numberwithin{equation}{section}

\usepackage{amssymb}
\usepackage[protrusion=true,expansion=true]{microtype}	
\usepackage{amsmath,amsfonts,amsthm} 
\usepackage[pdftex]{graphicx}	
\usepackage{url}
\usepackage[title,titletoc,page]{appendix} 
\usepackage{listings} 
\usepackage{color}
\usepackage{comment}
\usepackage{caption}
\usepackage{subcaption}
\usepackage{algorithm}

\usepackage{fancyhdr}
\numberwithin{equation}{section}		
\numberwithin{figure}{section}			
\numberwithin{table}{section}				

\newcommand{\babs}[1]{\left|{#1}\right|}
\newcommand{\bnorm}[1]{\left|\left|{#1}\right|\right|}

\newcommand{\vect}[1]{\boldsymbol{\mathbf{#1}}}

\usepackage[left=3.5cm,right=3.5cm,top=3.5cm,bottom=3.5cm,footskip=1.5cm,headsep=1cm]{geometry}
\usepackage[%
pdfauthor={Habib Ammari, Silvio Barandun, Jinghao Cao, Bryn Davies,
Erik Orvehed Hiltunen, and Ping Liu},%
pdftitle={Three-dimensional non-Hermitian skin effect},
pdfsubject={Three-dimensional non-Hermitian skin effect},
pdfkeywords={Non-Hermitian systems, skin effect, subwavelength resonators, imaginary gauge
potential},
]{hyperref}

\usepackage[T1]{fontenc} 
\usepackage{lmodern} 
\rmfamily 
\DeclareFontShape{T1}{lmr}{b}{sc}{<->ssub*cmr/bx/sc}{}
\DeclareFontShape{T1}{lmr}{bx}{sc}{<->ssub*cmr/bx/sc}{}
\usepackage{lipsum}
\usepackage{mathtools}
\usepackage{amsmath}
\usepackage{amssymb}
\usepackage{tikz}
\usetikzlibrary{matrix,positioning,decorations.pathreplacing,calc}
\usetikzlibrary{
decorations.text,%
decorations.markings,%
shadows}
\usepackage{adjustbox}
\usepackage{graphicx} 

\usepackage{caption}
\usepackage{subcaption}

\usepackage{dsfont} 
\usepackage[format=hang]{caption}

\usepackage[normalem]{ulem}

\newcommand{\abs}[1]{\lvert#1\rvert}
\usepackage{bm}

\usepackage[
style=numeric,
sorting=nty,
maxnames=99,
maxalphanames=5,
natbib=true,
sortcites]{biblatex}

\DeclareNameAlias{default}{family-given}

\graphicspath{{figures/}}

\AtEveryBibitem{
 \clearfield{url}
 \clearfield{issn}
 \clearfield{isbn}
 \clearfield{urldate}
 
 \ifentrytype{book}{
  \clearfield{pages}}{%
 }
}

\usepackage{xargs}                      
\usepackage[prependcaption,textsize=tiny,textwidth=3cm]{todonotes}
\newcommandx{\unsure}[2][1=]{\todo[linecolor=red,backgroundcolor=red!25,bordercolor=red,#1]{#2}}
\newcommandx{\change}[2][1=]{\todo[linecolor=blue,backgroundcolor=blue!25,bordercolor=blue,#1]{#2}}
\newcommandx{\info}[2][1=]{\todo[linecolor=OliveGreen,backgroundcolor=OliveGreen!25,bordercolor=OliveGreen,#1]{#2}}
\newcommandx{\improvement}[2][1=]{\todo[linecolor=black,backgroundcolor=black!25,bordercolor=black,#1]{#2}}
\newcommandx{\thiswillnotshow}[2][1=]{\todo[disable,#1]{#2}}
\usepackage{amsthm}
\usepackage[noabbrev,capitalize]{cleveref}
\crefname{proposition}{Proposition}{Propositions}
\crefname{equation}{}{}

\newtheorem{thm}{Theorem}[section]

\newtheorem{lemma}[thm]{Lemma}
\newtheorem{prop}[thm]{Proposition}
\newtheorem{definition}{Definition}[section]

\newtheorem{rem}[definition]{Remark}

\crefname{assumption}{Assumption}{Assumptions}
\crefname{definition}{Definition}{Definitions}
\crefname{corollary}{Corollary}{Corollaries}
\crefname{enumi}{item}{items}

\usepackage{fancyhdr}

\fancypagestyle{plain}{%
\fancyhead[C]{}
\fancyfoot[C]{}

}
\fancypagestyle{nosection}{%
\fancyhead[CE]{}
\fancyhead[CO]{}
\fancyfoot[CE,CO]{\thepage}
}

\DeclareMathOperator{\R}{\mathbb{R}}
\DeclareMathOperator{\C}{\mathbb{C}}

\renewcommand{\tilde}{\widetilde}
\renewcommand{\hat}{\widehat}
\renewcommand{\bar}[1]{\overline{#1}}

\renewcommand{\qedsymbol}{$\blacksquare$}

\newcommand{\D}{\mathcal{D}}

\newcommand{\K}{\mathcal{K}}
\renewcommand{\L}{\mathcal{L}}

\renewcommand{\O}{\mathcal{O}}

\newcommand{\p}{\partial}
\newcommand{\dx}{\mathrm{d}}

\renewcommand{\Re}{\mathrm{Re}}

\newcommand{\nm}{\noalign{\smallskip}}
\newcommand{\ds}{\displaystyle}

\usepackage{fancyhdr}

\fancypagestyle{plain}{%
\fancyhead[C]{}
\fancyfoot[C]{}

}
\fancypagestyle{nosection}{%
\fancyhead[CE]{}
\fancyhead[CO]{}
\fancyfoot[CE,CO]{\thepage}
}

\usepackage{amssymb}
\usepackage[protrusion=true,expansion=true]{microtype}	
\usepackage{amsmath,amsfonts,amsthm} 
\usepackage[pdftex]{graphicx}	
\usepackage{url}
\usepackage[title,titletoc,page]{appendix} 
\usepackage{listings} 
\usepackage{color}
\usepackage{comment}
\usepackage{caption}
\usepackage{subcaption}
\usepackage{algorithm}

\usepackage{fancyhdr}
\renewcommand{\epsilon}{\varepsilon}

\renewcommand{\tilde}{\widetilde}
\renewcommand{\hat}{\widehat}

\usepackage{tcolorbox}
\usetikzlibrary{tikzmark,calc}

\usepackage{enumerate}
\usepackage{upgreek}

\begin{document}

\title[The non-Hermitian skin effect with three-dimensional long-range coupling]{The non-Hermitian skin effect with three-dimensional long-range coupling}

 \author[H. Ammari]{Habib Ammari}
\address{\parbox{\linewidth}{Habib Ammari\\
 ETH Z\"urich, Department of Mathematics, Rämistrasse 101, 8092 Z\"urich, Switzerland}}
\email{habib.ammari@math.ethz.ch}
\thanks{}

\author[S. Barandun]{Silvio Barandun}
 \address{\parbox{\linewidth}{Silvio Barandun\\
 ETH Z\"urich, Department of Mathematics, Rämistrasse 101, 8092 Z\"urich, Switzerland}}
 \email{silvio.barandun@sam.math.ethz.ch}

\author[J. Cao]{Jinghao Cao}
 \address{\parbox{\linewidth}{Jinghao Cao\\
 ETH Z\"urich, Department of Mathematics, Rämistrasse 101, 8092 Z\"urich, Switzerland}}
\email{jinghao.cao@sam.math.ethz.ch}

\author[B. Davies]{Bryn Davies}
 \address{\parbox{\linewidth}{Bryn Davies\\
Department of Mathematics, Imperial College London, 180 Queen's Gate, London SW7~2AZ, UK}}
\email{bryn.davies@imperial.ac.uk}

 \author[E.O. Hiltunen]{Erik Orvehed Hiltunen}
\address{\parbox{\linewidth}{Erik Orvehed Hiltunen\\
Department of Mathematics, Yale University, 10 Hillhouse Ave,
New Haven, CT~06511, USA}}
\email{erik.hiltunen@yale.edu}

\author[P. Liu]{Ping Liu}
 \address{\parbox{\linewidth}{Ping Liu\\
 ETH Z\"urich, Department of Mathematics, Rämistrasse 101, 8092 Z\"urich, Switzerland}}
\email{ping.liu@sam.math.ethz.ch}

\maketitle

\begin{abstract}
We study the non-Hermitian skin effect in a three-dimensional system of finitely many subwavelength resonators with an imaginary gauge potential. 
We introduce a discrete approximation of the eigenmodes and eigenfrequencies of the system in terms of the eigenvectors and eigenvalues of the so-called gauge capacitance matrix $\mathcal C_{N}^{\gamma}$, which is a dense matrix due to long-range interactions in the system. Based on translational invariance of this matrix and the decay of its off-diagonal entries,  we prove the condensation of the eigenmodes at one edge of the structure by showing the exponential decay of its pseudo-eigenvectors. In particular, we consider a range-$k$ approximation to keep the long-range interaction to a certain extent, thus obtaining a $k$-banded gauge capacitance matrix $\mathcal C_{N,k}^{\gamma}$. Using techniques for Toeplitz matrices and operators, we establish the exponential decay of the pseudo-eigenvectors of $\mathcal C_{N,k}^{\gamma}$ and demonstrate that they approximate those of the gauge capacitance matrix $\mathcal C_{N}^{\gamma}$ well. Our results are numerically verified. In particular, we show that long-range interactions affect only the first eigenmodes in the system. As a result, a tridiagonal approximation of the gauge capacitance matrix, similar to the nearest-neighbour approximation in quantum mechanics, provides a good approximation for the higher modes. Moreover, we also illustrate numerically the behaviour of the eigenmodes and the stability of the non-Hermitian skin effect with respect to disorder in a variety of three-dimensional structures.

\end{abstract}

\date{}

\bigskip

\par

\bigskip

\\

\section{Introduction}

The non-Hermitian skin effect is the phenomenon whereby  a large proportion 
of  the bulk eigenmodes of a non-Hermitian system are localised at one edge of an open chain of resonators. It has applications in topological photonics, phononics, and other condensed matter systems \cite{ghatak.brandenbourger.ea2020Observation,skinadd1,skinadd2, skinadd3}. It is one of the most promising and exciting research topics in physics in recent years \cite{zhang.zhang.ea2022review, okuma.kawabata.ea2020Topological,lin.tai.ea2023Topological,yokomizo.yoda.ea2022NonHermitian,skinadd4,skinadd5,borgnia.kruchkov.ea2020Nonhermitian}.

In our recent work \cite{ammari2023mathematical}, the non-Hermitian skin effect in the subwavelength regime is studied using first-principle mathematical analysis for one-dimensional systems of subwavelength resonators. An imaginary gauge potential (in the form of a first-order directional derivative) is added inside the resonators to break Hermiticity. Explicit asymptotic expressions for the subwavelength eigenfrequencies and eigenmodes are derived using a gauge capacitance matrix formulation of the problem (which is a reformulation of the standard capacitance matrices that are commonplace in Hermitian subwavelength physics and electrostatics). Moreover, the exponential decay of eigenmodes and their accumulation at one edge of the structure (the non-Hermitian skin effect) is shown to be induced by the Fredholm index of an associated Toeplitz operator. In \cite{SkinStability}, the robustness of the non-Hermitian skin effect in one dimension with respect to random imperfections in the system is quantified.  Moreover,  the competition between the non-Hermitian skin effect and Anderson localisation is illustrated. 

The generalisation of our previous work to higher-dimensional systems is important.  Although it is well-known that the skin effect can be experimentally realised in three-dimensional systems \cite{experimental3d,zhang.zhang.ea2022review}, there is no mathematical framework to theoretically support the condensation effect in three dimensions.  In this paper, we consider three-dimensional systems of subwavelength resonators with imaginary gauge potentials. 
As in the one-dimensional case, by using an asymptotic methodology,
we can approximate the subwavelength eigenfrequencies and eigenmodes of a finite chain of resonators by the eigenvalues of a 
so-called  \emph{gauge capacitance matrix} $\mathcal C_N^{\gamma}$. The main difference compared to the one-dimensional case is that the three-dimensional gauge capacitance matrix is not tridiagonal. Its off-diagonal entries account for long-range interactions, which is inherent in three-dimensional systems. Nevertheless, inspired by the nearest-neighbour approximation in quantum mechanics, we consider a range-$k$ approximation to keep the long-range interaction to a certain extent, thus obtaining a $k$-banded gauge capacitance matrix $\mathcal C_{N,k}^{\gamma}$. Using techniques for the Toeplitz matrices and operators, we prove the exponential decay of the pseudo-eigenvectors of such $k$-banded matrix and demonstrate that they approximate those of the gauge capacitance matrix $\mathcal C_{N}^{\gamma}$ well. We numerically verify our results by showing that the long-range interactions affect only the first eigenmodes in the system. A tridiagonal (nearest-neighbour) approximation of the gauge capacitance matrix, provides a good approximation for the higher modes, which successively improves as we increase the range of interactions. Finally, we illustrate numerically the behaviour of the eigenmodes and the stability of the non-Hermitian skin effect with respect to disorder in a variety of structures. 

 
The paper is organised as follows. In Section \ref{sect2}, we present the mathematical setup of the problem.  Section \ref{sect3} is devoted to the derivations of the fundamental solution of a (non-reciprocal) Helmholtz operator with an imaginary gauge potential and the analysis of the associated layer potentials. In Section \ref{sect4} we introduce our discrete approximation of the problem which provides approximations of the eigenfrequencies and eigenmodes of a finite chain of subwavelength resonators in terms of the eigenvalues and eigenvectors of the gauge capacitance matrix. Given this discrete formulation, we study in Section \ref{sect5} the properties of the gauge capacitance matrix and show that it can be approximated by a Toeplitz matrix. In Section \ref{sect6},  
we prove the exponential decay of the pseudo-eigenvectors of the gauge capacitance matrix and deduce 
the condensation of the subwavelength eigenmodes at one edge of the structure. In Section \ref{sect7}, we numerically illustrate our main findings in this paper for a variety of subwavelength resonator systems. Moreover, we show the robustness of the skin effect with respect to changes in the position of the resonators. The paper ends with some concluding remarks. Appendix A is devoted to the characterisation of the spectra and pseudo-spectra of perturbed banded Toeplitz matrices.

\section{Problem setting}  \label{sect2}

We will consider an array of material inclusions that are identical in size and shape and which have an imaginary gauge potential in their interior. To be more precise, we let $D \subset [0,1) \times \mathbb{R}^2$ be a  simply connected, bounded domain of class $C^{1,s}$ for some $0<s<1$. Then, we consider a one-dimensional array given by $\mathcal{D} := \bigcup_{i=1}^N D_i$, where $D_i$ is the translated domain $D_i := D + (i,0,0)^\top$, with the superscript $\top$ denoting the transpose. 

A resonant frequency $\omega\in\C$ is such that $\Re(\omega)>0$ and there exists a non-trivial associated eigenmode $u$ solution to 
\begin{equation}
    \label{helmholtz}
    \begin{cases}
    \ds \Delta u + \omega^2u = 0 \ & \text{in } \R^3 \setminus \overline{\mathcal{D}}, \\
    \nm
    \ds \Delta u + \omega^2 u +\gamma \partial_1 u = 0 \ & \text{in } \mathcal{D}, \\
    \nm
    \ds  u|_{+} -u|_{-}  = 0  & \text{on } \p \mathcal{D}, \\
    \nm
    \ds \left.\delta \frac{\p u}{\p \nu} \right|_{+} - \left.\frac{\p u}{\p \nu} \right|_{-} = 0 & \text{on } \p \mathcal{D}, \\
    \nm
    \ds u \text{ satisfies an outgoing radiation condition}. &   
    \end{cases}
\end{equation}
Here,  $\nu$ is the outgoing normal to $\partial D$, $|_{+}$ and $|_{-}$ denote the limits from the outside and inside of $\mathcal{D}$. $\delta >0$ is a non-dimensional material contrast parameter that will be assumed to be small, such that the system is in a high-contrast regime. Finally, the first-order directional derivative with coefficient $\gamma \neq 0$ corresponds to an imaginary gauge potential within the inclusions. This term breaks the time-reversal symmetry of the problem and is the crucial mechanism responsible for the condensation effects that will be observed here. For the outgoing radiation condition (which reduces to the Sommerfeld radiation condition for $\omega$ real), we refer the reader to \cite{vodev,santosa,alexander}. 

We are interested in  the subwavelength regime. We look for the subwavelength resonant frequencies and their associated eigenmodes for the system of resonators $\mathcal{D}$ within this regime, which we characterise by satisfying $\omega \to 0$ as $\delta\to 0$. In order to approximate these subwavelength resonant frequencies and eigenmodes, we perform an asymptotic analysis in a high-contrast limit, given by $\delta\to 0$. This limit recovers subwavelength resonant frequencies and eigenmodes, while keeping the characteristic size of the resonators fixed. As it will be seen in \Cref{sect4}, it leads to a discrete approximation of (\ref{helmholtz}) in terms of a so-called gauge capacitance matrix. By investigating the Toeplitz structure of the gauge capacitance matrix in \Cref{sect5}, we will prove in \Cref{sect6} the condensation of the subwavelength eigenmodes at one edge of the structure $\mathcal{D}$. 

\section{Fundamental solutions and single layer potentials}
 \label{sect3}
    Let $G^\omega$ be the outgoing fundamental solution of the Helmholtz operator
       $\Delta + \omega^2.$ Let $\mathcal{S}_D^{\omega}$ be the associated single layer potential
       $$
       \mathcal{S}_D^{\omega}:  \phi \in L^2(\partial D)\mapsto \int_{\partial D}G^\omega(x-y)\phi(y)\mathrm{d}\sigma(y).
       $$
    We define the Neumann-Poincar\'e operator associated to $G^\omega$ to be the operator
    $\mathcal{K}_D^{\omega,*} : L^2(\partial D) \rightarrow  L^2(\partial D) $ given by
    \begin{equation} \label{NPop}
        \mathcal{K}_D^{\omega,*}[\phi](x)=\int_{\partial D} \frac{\partial}{\partial \nu_x} G^\omega (x-y) \phi(y) \mathrm{d} \sigma(y), \quad x \in \partial D, \phi \in L^2(\partial D).
    \end{equation}
We recall the following well-known results regarding $\mathcal{S}_D^{\omega}[\phi]$; see, for instance, \cite{nedelec,bookspectral}.  

\begin{lemma}[Jump conditions] \label{lemmsingle} For the single layer potential $\mathcal{S}_D^\omega$, it holds
    \begin{equation}
    \left.\mathcal{S}_D^\omega[\phi]\right|_{+}=\left.\mathcal{S}_D^\omega[\phi]\right|_{-}, \quad \phi \in L^2(\partial D). 
    \end{equation}
   The associated Neumann--Poincar\'e operator $\mathcal{K}_D^{\omega,*} $ satisfies
    \begin{equation}
    \left.\frac{\partial}{\partial \nu} \mathcal{S}_D^\omega[\phi]\right|_{ \pm}=\left( \pm \frac{1}{2} I+\mathcal{K}_D^{\omega, *}\right)[\phi], \quad \phi \in L^2(\partial D).
    \end{equation}
\end{lemma}

It is also well-known that $\mathcal{S}^0_D:L^2(\p D) \to H^1(\p D)$ is invertible. Here, $H^1$ is usual Sobolev space of square-integrable functions whose weak derivative is square integrable.  

We now turn to the operator with an imaginary gauge potential $\Delta+\omega^2 +\gamma \partial_{x_1}$. The following result provides a fundamental solution for this operator. 
\begin{prop}
    Introduce 
    \begin{equation}
    \label{green_skin}
    G_\gamma^\omega(x) = -\frac{ \mathrm{exp}({-{\gamma}x_1/{2}+\mathrm{i}\sqrt{\omega^2-\gamma^2/4}\abs{x}})}{4\pi \abs{x}}.        
    \end{equation}
    Then, $G_\gamma^\omega$ satisfies 
    \begin{equation}
    \label{pde_skin}
        \Delta G_\gamma^\omega +\omega^2  G_\gamma^\omega +\gamma \partial_{x_1} G_\gamma^\omega = \delta_0 \qquad  \text{in } \mathbb{R}^3.  
    \end{equation}
\end{prop}
\begin{proof}
Observe that $F_\gamma^\omega (x) =-\frac{ \mathrm{exp}({\mathrm{i}\sqrt{\omega^2-\gamma^2/4}\abs{x}})}{4\pi \abs{x}}$ is a fundamental solution of the Helmholtz operator $\Delta+\omega^2-\gamma^2/4$. Letting
    \begin{equation}
        G_\gamma^\omega(x) = F_\gamma^\omega(x) e^{-\gamma x_1/2},
    \end{equation}
    we have 
    \begin{align}
    \left(\Delta+\omega^2-\gamma^2/4\right)F_\gamma^\omega(x) &= e^{\gamma x_1/2}\left( \Delta + \omega^2 + \gamma \partial_{x_1} \right)G_\gamma^\omega(x) \\
    & = \delta_0(x),
    \end{align}
    which proves the claim.    
\end{proof}
Observe that $G^\omega_\gamma$ does not satisfy the vanishing boundary condition as $x\to \infty$; nevertheless, as we will use $G^\omega_\gamma$ to represent the solution inside the compact domains $D_i$, the behaviour of $G^\omega_\gamma$ as  $x\to \infty$ does not affect the solution.

We now define the single layer potential associated to $G_{\gamma}^\omega$ by
\begin{equation}
{\tilde{\mathcal{S}}}_{D,\gamma}^\omega:\phi \in L^2(\partial D)\mapsto \int_{\partial D}G_\gamma^\omega(x-y)\phi(y)\mathrm{d}\sigma(y).
\end{equation}
Analogously to \eqref{NPop}, we also define the Neumann-Poincar\'e operator associated to (\ref{green_skin}) by  
\begin{equation}
    \tilde{\mathcal{K}}_{D,\gamma}^{\omega,*} [\phi](x)=\int_{\partial D} \frac{\partial}{\partial \nu_x} G^\omega_{\gamma}(x-y) \phi(y) \mathrm{d} \sigma(y), \quad x \in \partial D, \phi \in L^2(\partial D).
\end{equation}
Since $G^\omega_\gamma$ has the same singularity as $G^\omega$ at the origin,  similarly to Lemma \ref{lemmsingle}, we can obtain the following jump conditions for the normal derivative of ${\tilde{\mathcal{S}}}_{D,\gamma}^\omega$ across $\partial D$.
\begin{lemma} \label{lem:jump}
The Neumann-Poincar\'e operator $\tilde{\mathcal{K}}_{D,\gamma}^{\omega,*} $ satisfies 
    \begin{equation}
    \left.\frac{\partial}{\partial \nu} {\tilde{\mathcal{S}}}_{D,\gamma}^\omega[\phi]\right|_{ \pm}=\left( \pm \frac{1}{2} I+{\tilde{\mathcal{K}}}_{D,\gamma}^{\omega, *}\right)[\phi].
    \label{jumpomega}
    \end{equation}
\end{lemma}

The following result will be of use to us. 
\begin{prop}
The operator  ${\tilde{\mathcal{S}}}_{D,\gamma}^0 : L^2(\partial D) \rightarrow L^2(\partial D)$ is injective and hence has a left inverse 
\end{prop} 
\begin{proof}
We assume that $\tilde{\mathcal{S}}_{D,\gamma}^0[\phi](x) = 0$ for $x\in \partial D$. For $x\in \R^3$, we define $u$ and $v$ as
$$u(x) = \tilde{\mathcal{S}}_{D,\gamma}^0[\phi](x) \quad \text{and} \quad v(x) =e^{\gamma x_1/2}u(x).$$
Since $\Delta u + \gamma \partial_{x_1} u = 0$, we find that $v$ satisfies the \emph{modified} Helmholtz equation
\begin{equation} \label{modhelm}
\begin{cases}\ds \Delta v - \frac{\gamma^2}{4} v = 0, \\ v|_{\partial D} = 0, \\ 
v(x) = O\left(|x|^{-1}\right) \quad \text{as } x\to \infty.
\end{cases}
\end{equation}
The modified Helmholtz equation is well-known to have a unique solution; indeed, if we multiply \eqref{modhelm} with $\overline{v}$ and integrate by parts, we have
$$\int_{\R^3}\left( |\nabla v|^2 + \frac{\gamma^2}{2}|v|^2\right) \dx \sigma =0,$$
which shows that $v(x) = 0$, and hence $u(x) = 0$. Finally, from \Cref{lem:jump}, we have that 
$$\phi = \frac{\partial u }{\partial \nu}\Big|_+ - \frac{\partial u }{\partial \nu}\Big|_- = 0,$$
which proves the claim.
\end{proof}


\section{Gauge capacitance matrix approximation} \label{sect4}

We are looking for non-trivial solutions to the non-reciprocal Helmholtz problem (\ref{helmholtz}) in the subwavelength regime. Problem (\ref{helmholtz}) can be reformulated using layer potential techniques. In fact, there exist density functions $\psi$ and $\phi$ such that the solution $u$ of (\ref{helmholtz}) can be represented as
\begin{equation}
\label{eq:ux_layerpotential}
    u(x)=
    \begin{cases}
        \tilde{\mathcal{S}}_{\mathcal{D},\gamma}^{\omega}[\psi], \ \ \ &x\in \mathcal{D},\\
        \mathcal{S}_\mathcal{D}^\omega[\phi], \ \ \ & x\in \mathbb{R}^3\backslash \mathcal{D}.
    \end{cases}
\end{equation}
Hence, for $\omega$ small enough, (\ref{helmholtz}) is equivalent to
\begin{equation}
    \mathcal{A}(\omega,\delta )\begin{pmatrix} \psi \\ \phi
    \end{pmatrix} = 0,
\end{equation}
with $\mathcal{A}(\omega,\delta ) : L^2(\partial \mathcal{D}) \times L^2(\partial \mathcal{D})  \rightarrow  H^1(\partial \mathcal{D}) \times L^2(\partial \mathcal{D})$ being given by
\begin{equation}
    \mathcal{A}(\omega,\delta ) = 
    \begin{pmatrix}
        \tilde{\mathcal{S}}_{\mathcal{D},\gamma}^{\omega} &  - \mathcal{S}_\mathcal{D}^{\omega}\\
        -\frac{1}{2}I+\tilde{\mathcal{K}}_{\mathcal{D},\gamma}^{\omega,*} & -\delta(\frac{1}{2}I+{\mathcal{K}}_{\mathcal{D}}^{\omega,*})
    \end{pmatrix}.
\end{equation}
Note that for $\omega = \delta = 0$, we have
\begin{equation}
\label{a00}
    \mathcal{A}(0,0) = 
    \begin{pmatrix}
        \tilde{\mathcal{S}}_{\mathcal{D},\gamma}^{0} &  - \mathcal{S}_\mathcal{D}^{0}\\
        -\frac{1}{2}I+\tilde{\mathcal{K}}_{\mathcal{D},\gamma}^{0,*} & 0
    \end{pmatrix}.
\end{equation}
In the next result, we show that $-\frac{1}{2}I+\tilde{\mathcal{K}}_{\mathcal{D},\gamma}^{0,*}$ has a non-trivial kernel. The main idea of the capacitance approximation is that the subwavelength resonant modes arise as perturbations of these kernel functions for small (but non-zero) $\delta$.  
\begin{lemma} \label{lemmakernel}
    If $\mathcal{D}$ consists of $N$ connected components  $D_i$, i.e., $ \mathcal{D}= \bigcup_{i=1}^N D_i$, then the kernel of the operator $-\frac{1}{2}I+\tilde{\mathcal{K}}_{\mathcal{D},\gamma}^{0,*}: L^2(\partial \mathcal{D}) \rightarrow L^2(\partial \mathcal{D})$ is $N$-dimensional and spanned by 
    \begin{equation}
        \psi_i = (\tilde{\mathcal{S}}_{\mathcal{D},\gamma}^{0} )^{-1}[\chi_{\partial D_i}],\ i=1,\ldots,N.
        \label{defpsi}
    \end{equation}
    Here (as usual), $\chi_{\partial D_i}$ denotes the indicator function of $\partial D_i$.
\end{lemma}
\begin{proof}
    Let $\psi_i$ be defined by (\ref{defpsi}). Introduce $ u = \tilde{\mathcal{S}}_{\mathcal{D},\gamma}^{0}[\psi_i]$. From (\ref{pde_skin}), the function $u$ satisfies the following equation:
    \begin{equation}
        \Delta u + \gamma \partial_1 u = 0 \quad \text{in } D_i.
        \label{eqindi}
    \end{equation}
    Moreover, $u=1$ on $\partial D_i$. Denote by $v(x) = u(x)- 1$ in $D_i$. Then, by integrating by parts over $D_i$, we have 
\begin{equation}
\label{eq:kerpsi}
\begin{split}
    0 & = \int_{D_i} \bigr(v\Delta v + \gamma v\partial_{x_1} v\bigr) \  \mathrm{d}x\\
    & = \int_{\partial D_i} v\frac{\partial v}{ \partial \nu} \ \mathrm{d}\sigma - \int_{D_i} \abs{\nabla v}^2 \ \mathrm{d}x+
    \frac{\gamma}{2} \int_{D_i} \partial_{x_1} v^2 \ \mathrm{d}x \\
    & =\int_{\partial D_i} v\frac{\partial v}{ \partial \nu} \ \mathrm{d}\sigma - \int_{D_i} \abs{\nabla v}^2 \ \mathrm{d}x +\frac{\gamma}{2} \int_{\partial D_i} \nu_1 v^2 \ \mathrm{d}\sigma, 
\end{split}
\end{equation}
where $\nu_1$ is the first component of the outgoing normal to $\partial D_i$. Since $v$ vanishes on $\partial D_i$, the first and third terms in \eqref{eq:kerpsi} are zero. We can conclude that $\nabla v=0$ and $v= 0$. It follows that $u = 1$ in $D_i$. Therefore, by the jump relation (\ref{jumpomega}), 
$$
 \left.\frac{\partial u}{\partial \nu}\right|_- = \left( -\frac{1}{2}I+\tilde{\mathcal{K}}_{\mathcal{D},\gamma}^{\gamma,*}\right)[\psi_i] =  0. 
$$
Conversely, assume that $\psi \in L^2(\partial \mathcal{D})$ satisfies
$$
\left( -\frac{1}{2}I+\tilde{\mathcal{K}}_{\mathcal{D},\gamma}^{\gamma,*}\right)[\psi] =  0.
$$
Then, $u= \tilde{\mathcal{S}}_{\mathcal{D},\gamma}^{0}[\psi]$ satisfies (\ref{eqindi}) in $\mathcal{D}$ together with the Neumann boundary condition $\left. \partial u/\partial \nu \right|_- =0$ on $\partial \mathcal{D}$. Let $F(x)= e^{\gamma x_1} \nabla u(x)$ for $x \in  \mathcal{D}$. Then $F$ satisfies the following curl-div system of equations in $ \mathcal{D}$ with a vanishing normal component on $\partial \mathcal{D}$: 
\begin{equation} \label{curldiv}
\begin{cases}
    \ds  \nabla \cdot F = 0 \ & \text{in } \mathcal{D}, \\
    \nm
    \ds  \nabla \times ( e^{-\gamma x_1} F) = 0 \ & \text{in } \mathcal{D}, \\
    \nm
    \ds F \cdot \nu  = 0 \ & \text{on } \partial \mathcal{D}.
    \end{cases}
\end{equation}
From the uniqueness of a solution to (\ref{curldiv}) (see, for instance, 
\cite[Sect. 2.2]{valli}), it follows that $u$ is constant inside each connected component $D_i$ of $\mathcal{D}$. Therefore, 
$$
   \psi_i = \sum_{i=1}^Nc_i (\tilde{\mathcal{S}}_{\mathcal{D},\gamma}^{0} )^{-1}[\chi_{\partial D_i}],
$$
for some constants $c_i$.
\end{proof}
The following result follows from Lemma \ref{lemmakernel}. 
\begin{lemma} \label{lemmakernel2}
    The kernel of the operator $-\frac{1}{2}I+\tilde{\mathcal{K}}_{\mathcal{D},\gamma}^{0}: L^2(\partial \mathcal{D}) \rightarrow L^2(\partial \mathcal{D})$ is spanned by 
    \begin{equation}
        \phi_i =  e^{\gamma x_1}\chi_{\partial D_i},\ i=1,\ldots,N.
    \end{equation}
\end{lemma}
\begin{proof}
For any $\psi\in L^2(\partial \mathcal{D})$,   by using the jump conditions (\ref{jumpomega}), we have
    \begin{equation}
        \begin{split}
            \left< \psi, \bigr(-\frac{1}{2} I+\tilde{\mathcal{K}}_{\mathcal{D},\gamma}^0\bigr) [\phi_i] \right> & = \int_{\partial D_i}e^{\gamma x_1} \left.\frac{\partial}{\partial \nu}\right|_-\tilde{\mathcal{S}}_{\mathcal{D},\gamma}^0[\psi]\ \mathrm{d}\sigma \\
            & = \int_{D_i} \bigg( e^{\gamma x_1} \Delta \tilde{\mathcal{S}}_{\mathcal{D},\gamma}^0[\psi] + \nabla(e^{\gamma x_1}) \cdot \nabla \tilde{\mathcal{S}}_{\mathcal{D},\gamma}^0[\psi]
             \bigg)\  \mathrm{d}x \\
            & = \int_{D_i} \bigg(e^{{\gamma x_1}}\Delta \tilde{\mathcal{S}}_{\mathcal{D},\gamma}^0[\psi]+\gamma e^{\gamma x_1}\partial_{x_1}\tilde{\mathcal{S}}_{\mathcal{D},\gamma}^0[\psi] \bigg)\ \mathrm{d}x \\
            & = 0,
        \end{split}
    \end{equation}
    where $\left< \; , \; \right>$ denotes the $L^2$ scalar product on $\partial D_i$. 
    
    On the other hand, since for $\gamma =0$, $-\frac{1}{2} I+ \mathcal{K}_{\mathcal{D}}^0: L^2(\partial \mathcal{D}) \rightarrow L^2(\partial \mathcal{D})$ is Fredholm of index zero and $\tilde{\mathcal{K}}_{\mathcal{D},\gamma}^0 - \mathcal{K}_{\mathcal{D}}^0: L^2(\partial \mathcal{D}) \rightarrow L^2(\partial \mathcal{D})$ is compact (as $G^0_\gamma$ and $G^0$ have the same singularity at the origin),  the operator $-\frac{1}{2} I+ \tilde{\mathcal{K}}_{\mathcal{D},\gamma}^0: L^2(\partial \mathcal{D}) \rightarrow L^2(\partial \mathcal{D})$ is also Fredholm of index zero. From Lemma \ref{lemmakernel}, it then follows that the dimension of the kernel of $-\frac{1}{2} I+ \tilde{\mathcal{K}}_{\mathcal{D},\gamma}^0$  is $N$. This completes the proof. 
\end{proof}


We now introduce the gauge capacitance matrix,  which allows us to reduce the problem of finding the  subwavelength eigenfrequencies and eigenmodes to a finite-dimensional eigenvalue problem. 
\begin{definition}[Gauge capacitance matrix] 
The gauge capacitance matrix  $\mathcal{C}_N^\gamma=\left((\mathcal{C}_N^\gamma)_{i,j}\right), i,j=1,\ldots, N,$ reads
\begin{equation} \label{capacitancedef}
    \left(\mathcal{C}_N^\gamma\right)_{i,j} = -\frac{\delta}{\int_{D_i}e^{\gamma x_1}\ \mathrm{d}x} \int_{\partial D_i} e^{\gamma x_1}(\mathcal{S}^0_{\mathcal{D}})^{-1}[\chi_{\partial D_j}](x) \, \mathrm{d}x.
\end{equation}
\end{definition}

The following results hold. 
\begin{thm}[Discrete approximations] \label{thm:approx}
\begin{enumerate}
\item[(i)] The $N$ subwavelength eigenfrequencies $\omega_i$ of (\ref{helmholtz}) satisfy, as $\delta \rightarrow 0$,
\begin{equation} \label{approxlambda}
\omega_n= \sqrt{\lambda_n} + O(\delta), 
\end{equation}
where $(\lambda_n)_{1\leq n\leq N}$ are the eigenvalues of $\mathcal{C}_N^\gamma$;
 \item[(ii)]   Let $v_n$ be the eigenvector of $\mathcal{C}_N^\gamma$ associated to $\lambda_n$. Then the normalised resonant mode $u_n$ associated to the resonant frequency $\omega_n$ is given by
\begin{equation} \label{eigenmodeapprox}
    u_n(x) = 
    \begin{cases}
        v_n \cdot \mathbf{S}^{\omega_n} (x) + O(\delta) \ \ &\text{in}\ \mathbb{R}^3\setminus \bar{\mathcal{D}}, \\
         v_n \cdot \tilde{\mathbf{S}}^{\omega_n}(x) + O(\delta) \ \ & \text{in}\ D_i,
    \end{cases}
\end{equation}
where 
\begin{equation}
\label{eigenvector:skin}
    \tilde{\mathbf{S}}^{\omega_n}(x) = \begin{pmatrix}
        \tilde{\mathcal{S}}_{\mathcal{D},\gamma}^{\omega_n} [\psi_1](x)\\ \ldots \\  \tilde{\mathcal{S}}_{\mathcal{D},\gamma}^{\omega_n} [\psi_N](x)
    \end{pmatrix},
\end{equation}
with $\psi_i = (\tilde{\mathcal{S}}_{\mathcal{D},\gamma}^0)^{-1}[\chi_{\partial D_i}]$ for $i=1,\ldots,N$.
\end{enumerate}
\label{thmgauge}
\end{thm}
\begin{proof} 
We make  use of the functional approach first introduced in \cite{haicmp} (see also \cite{ammari.davies.ea2021Functional}). 
In view of (\ref{a00}), by using Lemma \ref{lemmakernel},
we have the kernel basis functions of $\mathcal{A}(\omega,\delta)$ at $\delta = \omega = 0$:
\begin{equation}
\mathrm{ker}(\mathcal{A}(0,0))=\mathrm{span}\{\Psi_i\}_{i=1,\ldots,N},\ \Psi_i=\begin{pmatrix}
(\tilde{\mathcal{S}}_{\mathcal{D},\gamma}^{0} )^{-1}[\chi_{\partial D_i}] \\
(\mathcal{S}_{\mathcal D}^0)^{-1}[\chi_{\partial D_i}]
\end{pmatrix}.
\end{equation}
For $\omega$ in a small neighbourhood $\mathcal{V}$ of $0$, we have the following pole-pencil decomposition:
\begin{equation}
    \left(-\frac{1}{2} I +\tilde{\mathcal{K}}_{\mathcal{D},\gamma}^{\omega,*}\right)^{-1} = \sum_{j=1}^N \frac{\left<\phi_j,\cdot\right>\psi_j}{\omega^2} + \mathcal{R}(\omega), 
\end{equation}
where the right- and left-kernel functions $\psi_i$ and $\phi_i$ are normalised such that 
\begin{equation}
\begin{split}
        \frac{1}{\omega^2}\left<\phi_j,\left(-\frac{1}{2}+\tilde{\mathcal{K}}_{\mathcal{D},\gamma}^{\omega,*}\right)[\psi_i] \right> = \delta_{ij},
\end{split}
\end{equation}
and the operator-valued function $\omega \mapsto \mathcal{R}(\omega)$ is holomorphic on $\mathcal{V}$. Here, we can take $(\phi_j)_{j=1}^N$ 
\begin{equation}
     \phi_j=C_j
e^{\gamma x_1}\chi_{\partial D_j},
\end{equation}
for normalisation factors $C_j$. We have 
\begin{equation}
\begin{split}
        \left<\phi_j,\left(-\frac{1}{2}+\tilde{\mathcal{K}}_{\mathcal{D},\gamma}^{\omega,*}\right)[\psi_i] \right> & = \int_{\partial D_j} C_j e^{\gamma x_1} \left.\frac{\partial}{\partial \nu} \right|_-\tilde{\mathcal{S}}_{\mathcal{D},\gamma}^\omega [\psi_i] \ \mathrm{d}\sigma \\
        & = C_j \int_{D_j} \bigg(e^{{\gamma x_1}}\Delta \tilde{\mathcal{S}}_{\mathcal{D},\gamma}^\omega [\psi_i]+\gamma e^{\gamma x_1}\partial_{x_1}\tilde{\mathcal{S}}_{\mathcal{D},\gamma}^\omega[\psi_i] \bigg)\ \mathrm{d}x \\
        & = -\delta_{ij} {\omega^2} C_j \int_{D_j} e^{\gamma x_1} \ \mathrm{d}x. 
\end{split}
\end{equation}
It then follows that 
\begin{equation}
\ds    \phi_j = -\frac{1}{\int_{D_j}e^{\gamma x_1}\ \mathrm{d}x}e^{\gamma x_1}\chi_{\partial D_j}. 
\end{equation}
Hence we have the kernel functions
\begin{equation}
    \mathrm{ker}(\mathcal{A}^*(0,0))=\mathrm{span}\{\Phi_i\}_{i=1,\ldots,N},\quad \Phi_i=-\frac{1}{\int_{D_i}e^{\gamma x_1}\ \mathrm{d}x}\begin{pmatrix}
0 \\
e^{\gamma x_1}\chi_{\partial D_i}
\end{pmatrix}.   
\end{equation}
For small $\delta$, we can write
\begin{equation}
    \mathcal{A}(\omega,\delta) = \mathcal{A}(\omega,0)+\mathcal{L}(\omega,\delta).
\end{equation}
We have 
\begin{equation}
0 = \mathcal{A}(\omega,0)^{-1}\mathcal{A}(\omega,\delta)[\Phi] = \mathcal{A}(\omega,0)^{-1}(\mathcal{A}(\omega,0)+\mathcal{L})[\Phi].
\end{equation}
This leads to
\begin{equation}
    \left(I+\frac{K\mathcal{L}}{\omega^2}+\mathcal{R}\mathcal{L}\right)[\Phi] = 0,
\end{equation}
or equivalently, 
\begin{equation}
    (\omega^2I +K\mathcal{L}+\omega^2\mathcal{R}\mathcal{L})[\Phi] = 0,
\end{equation}
where $\ds K= \sum_{i=1}^N \left< \phi_i, \cdot \right> \psi_i$. 
Since $\mathcal{L} = \mathcal{L}_0+\hat{L}$ with $\mathcal{L}_0$ being independent of $\omega$ and $\hat{L}$ satisfying (in the operator norm) $\|\hat{L}\| = O(\omega \delta)$ uniformly in $\omega \in \mathcal{V}$, the subwavelength resonances are approximately  the square roots of the eigenvalues of $-K\mathcal{L}_0$ restricted to $ \mathrm{ker} (\mathcal{A}(0,0))$. This is given by the gauge capacitance matrix $\mathcal{C}_N^\gamma$. Finally, (\ref{eigenvector:skin}) follows from the representation formula (\ref{eq:ux_layerpotential}).
\end{proof}

\begin{rem}
\Cref{thm:approx} can easily be generalised to a setting with general material parameters. Assume that the resonator $D_i$ has wave  speed  $v_i$ and contrast parameter $\delta_i$. Correspondingly, let $v$ be the wave speed in the surrounding medium. Assume that $\delta_i = O(\delta)$, for $i=1,\ldots, N,$ and $0< \delta \ll 1$. For such a system of resonators, the resonance problem reads 
\begin{equation}
    \label{helmholtzg}
    \begin{cases}
    \ds \Delta u + \frac{\omega^2}{v^2}u = 0 \ & \text{in } \R^3 \setminus \overline{\mathcal{D}}, \\
    \nm
    \ds \Delta u + \frac{\omega^2}{v_i^2} u +\gamma \partial_1 u = 0 \ & \text{in } {D_i}, \quad i=1,\ldots, N, \\
    \nm
    \ds  u|_{+} -u|_{-}  = 0  & \text{on } \p \mathcal{D}, \\
    \nm
    \ds \left.\delta_i \frac{\p u}{\p \nu} \right|_{+} - \left.\frac{\p u}{\p \nu} \right|_{-} = 0 & \text{on } \p {D_i}, \quad i=1,\ldots, N,  \\
    \nm
    \ds   u \text{ satisfies an outgoing radiation condition.} &   
    \end{cases}
\end{equation}
It can be easily seen that the gauge capacitance matrix  $\mathcal{C}_N^\gamma=\left((\mathcal{C}_N^\gamma)_{i, j}\right), i,j=1,\ldots, N,$ defined by 
\begin{equation} \label{newC}
\left(\mathcal{C}_N^\gamma\right)_{i, j} = -\frac{\delta_i v_i^2}{\int_{D_i}e^{\gamma x_1}\ \mathrm{d}x} \int_{\partial D_i} e^{\gamma x_1}(\mathcal{S}^0_{\mathcal{D}})^{-1}[\chi_{\partial D_j}] \mathrm{d}x,
\end{equation}
provides as in Theorem \ref{thmgauge} leading-order approximations
to the  $N$ subwavelength resonant frequencies  of (\ref{helmholtzg}) and their associated resonant eigenmodes. In fact, the $N$ subwavelength resonant frequencies of (\ref{helmholtzg}) and their associated eigenmodes are respectively approximated by (\ref{approxlambda}) and (\ref{eigenmodeapprox}), where $(\lambda_n)_{1\leq n\leq N}$ and $(v_n)_{1\leq n\leq N}$ are respectively the eigenvalues and associated eigenvectors of  $\mathcal{C}_N^\gamma$ defined by (\ref{newC}). 
\end{rem}

\section{Toeplitz structure of the gauge capacitance matrix} \label{sect5}

In this section, in order to demonstrate the non-Hermitian skin effect in three-dimensional systems of subwavelength resonators,  we derive some  properties of the gauge capacitance matrix $\mathcal C_{N}^{\gamma}$. In particular, we  show that it is, up to a small perturbation, a Toeplitz matrix. In order to do that, we need to establish some convergence results, characterising the matrix in the limit when the number of resonators goes to infinity. This is a generalisation of the approach used in \cite{ammari2023spectral}. In particular, we look at periodic chain structures with only one single resonator in the unit cell. This structure is equivalent to equidistant chains of resonators. Formally, we let $l_1 \in \mathbb{R}^3$ denote the lattice vector generating the linear lattice 
\begin{equation}
    \Lambda : = \{ml_1 | m\in\mathbb{Z}\}.
\end{equation}
Without loss of generality, assume that $l_1$ is aligned with the $x_1$-axis. Denote by $Y\subset \mathbb{R}^3$ the fundamental domain of the lattice. The dual lattice of $\Lambda$, denoted by $\Lambda^*$, is generated by $\alpha_1$, satisfying $\alpha_1\cdot l_1 = 2\pi$. The Brillouin zone $Y^*$ is defined as
\begin{equation}
    Y^* :  = (\mathbb{R} \times \{0_{\mathbb{R}^2}\})/ \Lambda^*.
\end{equation}
As before, we let $D$ denote the periodically repeated resonator $D$ consider an infinite structure of resonators
\begin{equation}
    \bigcup_{m \in \Lambda} \big( D +m \big). 
\end{equation}
For any function $f\in L^2(\mathbb{R}^3)$, the Floquet transform of $f$ is defined as 
\begin{equation}
    \mathcal{F}[f](x, \alpha):=\sum_{m \in \Lambda} f(x-m) e^{\mathrm{i} \alpha \cdot m}, \quad x, \alpha \in \mathbb{R}^3.
\end{equation}
Extending the set of equations (\ref{helmholtz}) to the infinite structure and applying the Floquet transform with $u^\alpha(x):=\mathcal{F}[u](x,\alpha)$, we obtain
\begin{equation}
\label{eq:periodicskin}
    \begin{cases}\Delta u^\alpha+\omega ^2 u^\alpha=0 & \text { in } Y \backslash D, \\ \Delta u^\alpha+\omega^2 u^\alpha +\gamma \partial_1 u^\alpha =0 & \text { in } D,\\ \left.u^\alpha\right|_{+}-\left.u^\alpha\right|_{-}=0 & \text { on } \partial D, \\ \left.\ds \delta \frac{\partial u^\alpha}{\partial \nu}\right|_{+}-\left. \ds\frac{\partial u^\alpha}{\partial \nu}\right|_{-}=0 & \text { on } \partial D, \\ u^\alpha\left(x_1, x_2, x_3\right) & \text { is } \alpha \text {-quasiperiodic in } x_1, \\ u^\alpha\left(x_1, x_2,x_3\right) & \text{ satisfies the $\alpha$-quasiperiodic outgoing radiation} \\ &\hspace{0.5cm} \text{ condition as } \sqrt{x_2^2+x_3^2} \rightarrow \infty.\end{cases}
\end{equation}
We refer to \cite{ammari2018mathematical} for a precise statement of the $\alpha$-quasiperiodic radiation condition. 
We can proceed as in the usual case and define the quasiperiodic function outside the resonators as
\begin{equation}\label{eq:Galph}
    G^{\alpha, k}(x, y):=-\sum_{m \in \Lambda} \frac{e^{\mathrm{i} k|x-y-m|}}{4 \pi|x-y-m|} e^{\mathrm{i} \alpha \cdot m},
\end{equation}
together with the quasiperiodic single layer potential
\begin{equation}
    \mathcal{S}_D^{\alpha, k }[\phi](x):=\int_{\partial D} G^{\alpha, k}(x, y) \phi(y) \mathrm{d} \sigma(y), \quad x \in \mathbb{R}^3,
\end{equation}
for $\phi \in L^2(\partial D)$. 
The series in \eqref{eq:Galph} converges uniformly for $x$ and $y$ in compact sets of $\R^3$, $x\neq y$,  and $k \neq |\alpha + q|$ for all $q\in \Lambda^*$.
Repeating the procedures in Section \ref{sect3}, we can generalise the concept of quasiperiodic capacity first introduced in \cite{bandalpha} to (\ref{eq:periodicskin})  and define the gauge capacitance matrix as follows.
\begin{definition}
    If $\alpha \neq 0$, the gauge quasiperiodic capacity is  given by 
    \begin{equation}
    \hat{\mathcal{C}}^{\alpha,\gamma} = -\frac{\delta}{\int_{D}e^{\gamma x_1}\ \mathrm{d}x}\int_{\partial D} e^{\gamma x_1}(\mathcal{S}^{\alpha,0}_D)^{-1}[\chi_{\partial D}](x)\,  \mathrm{d}\sigma(x).
\end{equation}   
\end{definition}
The above introduced gauge capacity is a ``dual-space'' capacitance representation. The infinitely periodic system with only one resonator in a unit cell is equivalent to an infinite chain of equidistant resonators which has a ``real-space'' capacitance representation through the inverse Floquet-Bloch transformation (for details, see \cite{finiteinfinite, ammari2023spectral}) given by
\begin{equation}
\label{defrealspace}
   \mathcal{C}^\gamma_{i,j} = \frac{1}{\babs{Y^*}} \int_{Y^*} \hat{\mathcal{C}}^{\alpha,\gamma} e^{-\mathrm{i}\alpha (i-j)} \mathrm{d}\alpha.
\end{equation}
It is clear from the definition above that $\mathcal{C}^\gamma$ is a Toeplitz matrix. 
We now circle back to the finite system. Let $\mathcal{D}_N$ denote a finite chain of $N$ equidistant resonators and $\mathcal{C}_N^\gamma$ be the corresponding gauge capacitance matrix. We can extend $\mathcal{D}_N$ to a larger chain $\mathcal{D}_R$ with $R$ resonators and denote by $\mathcal{C}_{R}^\gamma$ the corresponding gauge capacitance matrix. We then let $\tilde{\mathcal{C}}_{N}^\gamma$ be the  embedded $N\times N$-size centre block of $\mathcal{C}_{R}^\gamma$. We call $\tilde{\mathcal{C}}_{N}^\gamma$ the truncated gauge capacitance matrix with size $N\times N$. Following the same steps as those in \cite{finiteinfinite}, we have the following theorem.
\begin{thm}
\label{thm:finiteinfinite}
    For $R\rightarrow \infty$, we have for every $i,j=1,\ldots,N,$ that
    \begin{equation}
        \lim_{R\to\infty} (\tilde{\mathcal{C}}_N^\gamma)_{i,j} = \mathcal{C}^{\gamma}_{i,j}.
    \end{equation}
\end{thm}
\begin{proof}
As shown in \cite{finiteinfinite}, for given $\phi \in L^2(D)$, we have
\begin{equation}
    \mathcal{S}_D^\alpha [\phi] = \mathcal{S}_{\mathcal{D}_R}[\phi] + \mathcal{R}^\alpha[\phi],
\end{equation}
where $\mathcal{R}^\alpha = o(1)$ as $R\rightarrow \infty$. From the Neumann series, we have
\begin{equation}
    (\mathcal{S}_D^\alpha)^{-1}[\chi_{\partial D}] = \mathcal{S}_{\mathcal{D}_R}^{-1}[\chi^\alpha]+ o(1),
\end{equation}
    with
    \begin{equation}
        \chi^\alpha=\sum_{\abs{m}< Nl_1} \chi_{\partial D_i} e^{\mathrm{i} \alpha \cdot m}.
    \end{equation}
    Finally, we have as $N \rightarrow \infty$ that
    \begin{align}
        \mathcal{C}_{ij}^\gamma & = - \frac{1}{\babs{Y^*}} \int_{Y^*} \frac{\delta}{\int_{D}e^{\gamma x_1}\ \mathrm{d}x}\int_{\partial D} e^{\gamma x_1}(\mathcal{S}^{\alpha,0}_D)^{-1}[\chi^\alpha ]  e^{-\mathrm{i}\alpha (i-j)} \mathrm{d} \sigma \mathrm{d}\alpha + o(1)\nonumber  \\
        & = \frac{\delta}{\int_{D_i}e^{\gamma x_1} \mathrm{d} x} \int_{\partial D_i}e^{\gamma x_1} \mathcal{S}_{\mathcal{D}_R}^{-1}[\chi_{\partial D_j}] \mathrm{d}\sigma + o(1)   \\
        & = (\tilde{\mathcal{C}}_N^\gamma)_{ij} +o(1).\nonumber
    \qedhere
    \end{align}
\end{proof}

The next results show that the matrix $\mathcal{C}_N^\gamma$ has a Toeplitz structure up to some small perturbation.
\begin{lemma}
\label{lemma:estimate}
We have the following estimate of the matrix entries:
\begin{equation}
    \babs{(\mathcal{C}_N^\gamma)_{i,j}-(\tilde{\mathcal{C}}_{N}^\gamma)_{i,j}} \leq \frac{\delta K}{(1+\min(i,N-i))(1+\min(j,N-j))},
\end{equation}
for some constant $K$ independent of $N$.
\end{lemma}
\begin{proof}
    Defining $x_1^i$ to be the $x_1$-coordinate of the center of $D_i$, we can compute
       \begin{equation}
    \begin{split}  
            \babs{(\mathcal{C}_N^\gamma)_{i,j}-(\tilde{\mathcal{C}}_{N}^\gamma)_{i,j}} &=\babs{\frac{\delta }{\int_{D_i}e^{\gamma x_1}\mathrm{d}x}\int_{\partial D_i}e^{\gamma x_1}\left(\mathcal{S}_{\mathcal{D}_N}^{-1}-\mathcal{S}_{\mathcal{D}_R}^{-1}\right)[\chi_{\partial D_{j}}]\mathrm{d}\sigma} \\
             & = \babs{\frac{\delta}{\int_{D_i}e^{\gamma (x_1-x_1^i)}\mathrm{d}x} \int_{\partial D_i} e^{\gamma(x_1 - x_1^i)} \left(\mathcal{S}_{\mathcal{D}_N}^{-1}-\mathcal{S}_{\mathcal{D}_R}^{-1}\right)[\chi_{\partial D_j}]  \mathrm{d}\sigma}  \\
             &\leq \frac{\delta K}{(1+\min(i,N-i))(1+\min(j,N-j))},
    \end{split}
    \end{equation}  
    for some constant $K$ independent of $N$. Here, the last inequality follows from the proof of \cite[Lemma 3.1]{ammari2023spectral}.
\end{proof}

\begin{lemma}\label{lem:capacitancetotoeplitz2}
We have the following estimate:
\begin{equation}
\label{eq:almosttopelitz}
 \babs{(\mathcal{C}_N^\gamma)_{i,j}-\mathcal{C}_{i,j}^\gamma} \leq \frac{\delta K}{(1+\min(i,N-i))(1+\min(j,N-j))}, \ i,j =1,\ldots, N,
\end{equation}
for some constant $K$ independent of $N$.
\end{lemma}
\begin{proof}
    From Theorem \ref{thm:finiteinfinite}, it follows that for every $\varepsilon>0$, there exists a large enough $R$ so that $\babs{(\tilde{\mathcal{C}}_{N}^\gamma)_{i,j}-\mathcal{C}_{i,j}^\gamma}<\varepsilon$. The desired estimate follows from Lemma \ref{lemma:estimate} and the triangular inequality.
\end{proof}

The following classical result allows us to describe the decay property of the entries of the gauge capacitance matrix. We  
refer, for instance, to \cite{harmonicanalysis} for its proof. 
\begin{lemma} \label{thm:RLlemma}
If $f\in BV(\mathbb{T})$, then for the $n$-th Fourier coefficient $\hat{f}(n)$ it holds
    \begin{equation}
        |\hat{f}(n)| \leq \frac{\operatorname{Var}_{\mathbb{T}}(f)}{2 \pi|n|}, \quad n \neq 0.
    \end{equation}
    Here, $\mathbb{T}$ is the unit circle, $BV(\mathbb{T})$ is the set of functions of bounded variation on $\mathbb{T}$, and 
    $\operatorname{Var}_{\mathbb{T}}(f)$ is the total variation of $f$. 
\end{lemma}
Finally, from Lemma \ref{thm:RLlemma} we obtain the following decay estimate for the entries of the gauge capacitance matrix. 
\begin{prop}
    For $i\neq j$, there exists a constant $K$, independent of $N$, so that
    \begin{equation}\label{equ:capacitancematrixesitmate1}
        \abs{\mathcal{C}^\gamma_{i, j}} \leq  \frac{\delta K}{\abs{i-j}}.
    \end{equation}
\end{prop}

\begin{proof} 
 We first note that by exactly the same arguments as those in \cite[Section 3.3]{alphalipschitz},
 we can show that $\alpha\mapsto \hat{\mathcal{C}}^{\alpha, \gamma}$
 is piecewise differentiable whose derivative is absolutely integrable over $Y^*$. It follows that it has bounded variation.
 By definition (\ref{defrealspace}), $\mathcal{C}^\gamma_{i,j}$ is the $(i-j)$-th Fourier coefficient of $\hat{\mathcal{C}}^{\alpha, \gamma}$, so \Cref{thm:RLlemma} shows that
    \begin{equation} \label{decayC}
        \babs{\mathcal{C}^{\gamma}_{i, j}}\leq  \frac{\delta K}{\babs{i-j}}. \qedhere
    \end{equation} 
\end{proof}
Finally, we generalise this decay estimate to the finite gauge capacitance matrix. We will use the following classical notion of the (regular) capacitance matrix.
\begin{definition}
The capacitance matrix $C_N = (C_N)_{i,j}$, for  $i,j = 1,\ldots,N$, is defined by
    \begin{equation}
    (C_N)_{i,j} = -\int_{\partial D_i} (\mathcal{S}^{0}_{\D_N})^{-1}[\chi_{\partial D_j}](x)\,  \mathrm{d}\sigma(x).
\end{equation}   
\end{definition}
It is well-known that the diagonal coefficients $(C_N)_{i,i}$ are positive while the off-diagonal coefficients $(C_N)_{i,j}, i\neq j$ are negative \cite{ammari2018mathematical}. The capacitance coefficients are equivalently given by  (see, for example, \cite{ammari.davies.ea2021Functional})
\begin{equation}\label{eq:C}
    (C_N)_{i,j} = -\int_{\partial D_i}\frac{\partial V_j(x)}{\partial \nu}\,  \mathrm{d}\sigma(x),
\end{equation}
where $V_j$ is given by the solution to the following problem:
\begin{equation}\label{eq:Vj}\begin{cases}
\Delta V_j = 0,  & x\in \R^3 \setminus \D_N, \\ 
V_j(x) = \chi_{\partial D_i}(x), &x\in \partial \D_N,\\
V_j(x) = O(|x|^{-1}), & x\to \infty.
\end{cases}\end{equation}
We begin with the following lemma, which states that the capacitance coefficient associated to two given resonators can only increase if we add additional resonators.
\begin{lemma}\label{lem:increase}
Let $\D_N$ and $\D_{N+1}$ be collections of $N$ and $N+1$ resonators, such that $\D_{N+1} = \D_N \cup D_{N+1}$.  For all $1<i,j<N$ we then have
$$(C_N)_{i,j} \leq (C_{N+1})_{i,j}.$$
\end{lemma}
\begin{proof}
Let $V_{N,j}$ and $V_{N+1,j}$, respectively, be the unique solutions to the problems
$$\begin{cases}
\Delta V_{N,j} = 0,  & x\in \R^3 \setminus \D_N, \\ 
V_{N,j}(x) = \chi_{\partial D_j}(x), &x\in \partial \D_N,\\
V_{N,j}(x) = O(|x|^{-1}), & x\to \infty,
\end{cases} \qquad 
\begin{cases}
\Delta V_{N+1,j} = 0,  & x\in \R^3 \setminus \D_{N+1}, \\ 
V_{N+1,j}(x) = \chi_{\partial D_j}(x), & x\in \partial \D_{N+1}, \\
V_{N+1,j}(x) = O(|x|^{-1}), & x\to \infty.
\end{cases}
$$
From the maximum principle, we observe that $0<V_{N,j}(x),V_{N+1,j}(x)<1$. Then we define $u(x) = V_{N,j}(x) - V_{N+1,j}(x)$. Note that $u$ satisfies
$$
\begin{cases}
\Delta u = 0,  & x\in \R^3 \setminus \D_{N+1}, \\ 
u(x) = 0, & x\in \partial \D_{N}, \\
0<u(x) < 1, & x\in \partial D_{N+1}, \\
u(x) = O(|x|^{-1}), & x\to \infty.
\end{cases}$$
Again, by the maximum principle we have $0<u(x)<1$ in $\R^3 \setminus \D_{N+1}$. Together with the boundary condition $u(x)=0$ on $\partial \D_{N}$, this implies that  
$$\frac{\partial u}{\partial \nu}  \geq 0, \quad x\in \partial \D_{N}.$$
From \eqref{eq:C}, we then have 
$$(C_{N+1})_{i,j} - (C_{N})_{i,j} = \int_{\partial D_i}\frac{\partial u}{\partial \nu} \dx \sigma \geq 0,$$
which proves the claim.
\end{proof}
We now have the following decay estimate of the gauge capacitance coefficients.
\begin{prop}\label{lem:capacitancematrixesitmate1}
    Let $d_{i,j}$ be the minimal distance between $D_i$ and $D_j$. For $i\neq j$, there exists a constant $K$, depending only on the shape of $D_i$ and $D_j$, such that $\abs{(\mathcal{C}_N)_{i, j}^\gamma} \leq  \delta Kd_{i,j}^{-1}$. In particular, we have a constant $K'$, independent of $N$, such that 
    \begin{equation}\label{equ:decayofcapamatrixelements1}
        \abs{(\mathcal{C}_N)_{i, j}^\gamma} \leq  \frac{\delta K'}{\abs{i-j}}.
    \end{equation}
\end{prop}
\begin{proof}
We begin by observing that
    \begin{equation}
    \begin{split}
     \babs{(\mathcal{C}_N)_{i,j}^\gamma}
    &= \babs{\frac{\delta_i v_i^2}{\int_{D_i}e^{\gamma x_1} \mathrm{d}x}\int_{\partial D_i}e^{\gamma x_1}\psi_{j}\mathrm{d}\sigma}  \\ 
    &\leq \delta K_1 \int_{\partial D_i}|\psi_{j}|\mathrm{d}\sigma,
    \end{split}\label{eq:est1}
    \end{equation}
    for some constant $K_1$. We moreover have 
    $$\psi_j = \left(\frac{1}{2} + \K_{\D_N}^{0,*}\right)[\psi_j] = \frac{\partial}{\partial \nu} V_j(x),$$
    where $V_j$ is defined as in \eqref{eq:Vj}. In the same way as in the proof of \Cref{lem:increase}, the maximum principle and boundary conditions implies that $\partial_\nu V_j \geq 0$ on $\partial \D_N \setminus \partial D_j$. From \eqref{eq:est1}, it therefore follows that 
    \begin{equation}\label{eq:est2}
    \babs{(\mathcal{C}_N)_{i,j}^\gamma} \leq \delta K_1 \babs{(C_N)_{i,j}},
    \end{equation}
    where $(C_N)_{i,j}$ are the regular capacitance coefficients, defined as in \eqref{eq:C}. Let now $\D_2 = D_i \cup D_j$ be the resonator structure consisting solely of the two components $D_i$ and $D_j$, and let $C_2$ be the $2\times 2$ capacitance matrix associated to $\D_2$. Following \cite[Lemma 4.3]{ammari2020topologically}, we then have 
    $$|(C_2)_{1,2}| < \frac{K_2}{d_{i,j}},$$
    for some constant $K_2$. Moreover, since $(C_N)_{i,j}$ and $(C_2)_{1,2}$  are both negative, we have from \Cref{lem:increase} that 
    $$|(C_N)_{i,j}| \leq |(C_2)_{1,2}| < \frac{K_2}{d_{i,j}}.$$
    This, combined with \eqref{eq:est2}, proves the claim.
\end{proof}

    Figure \ref{fig:decay} shows the decay rate of the entries of the  gauge capacitance matrix, which is in accordance with Proposition \ref{lem:capacitancematrixesitmate1}. We remark that the decay rate is close to $|i-j|^{-r}$ for $r\approx 2$. The precise decay rate depends on the geometry of the resonators. In the dilute limit (when the resonators are asymptotically small), the decay rate will approach $r=1$ \cite{ammari2020topologically}. Here, we use the multipole method, as described in \cite[Appendix C]{bandalpha}, to compute the associated gauge capacitance matrix.
\begin{figure}[!h]
    \centering
    \includegraphics[width = 0.8\textwidth]{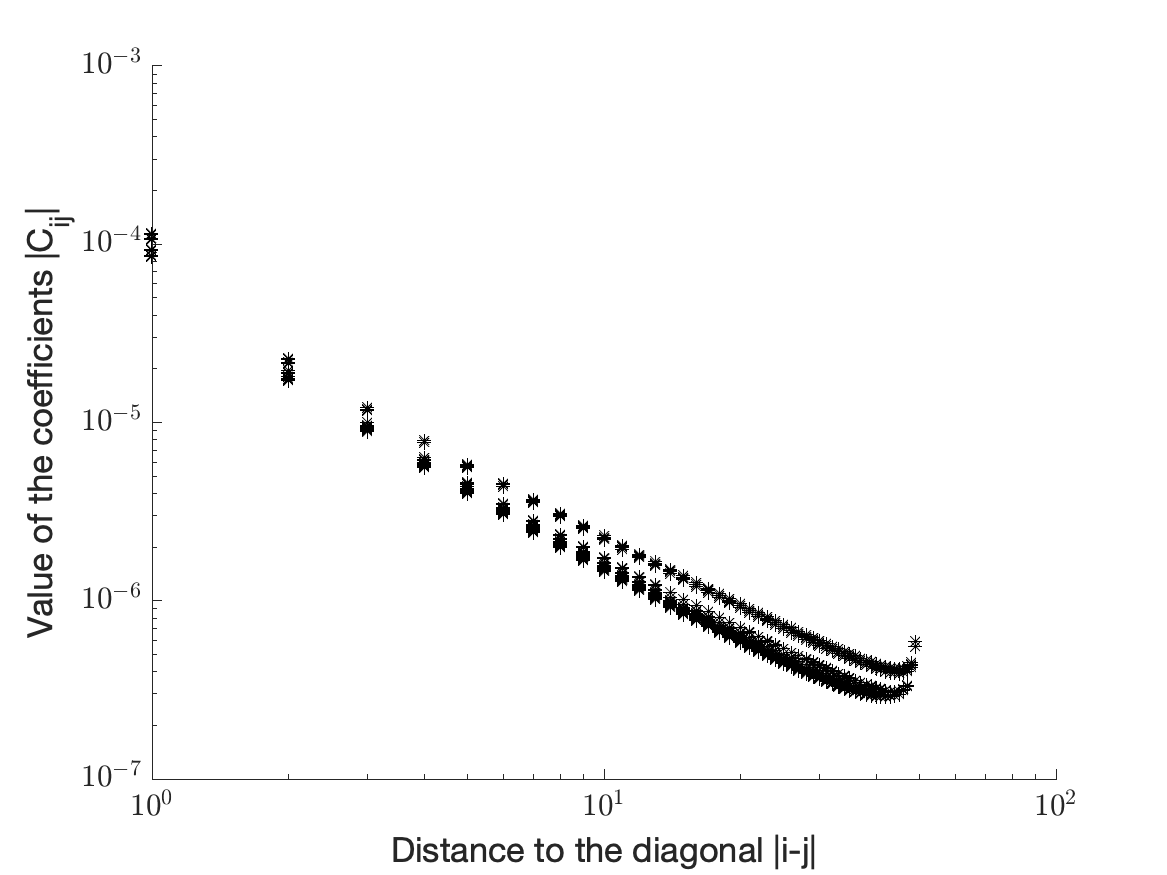}
    \caption{Decay of the entries of the gauge capacitance matrix for $N=100$ and $\gamma = 1$ in $log$-$log$ scale. The $x$-axis represents the index $\babs{i-j}$ and the $y$-axis displays the absolute value of $(C_N^\gamma)_{i,j}$.}
    \label{fig:decay}
\end{figure}

\section{Non-Hermitian skin effect} \label{sect6}
 As shown in item (ii) of Theorem \ref{thmgauge}, at leading-order in $\delta$, each subwavelength resonant mode $u_{n}$ of (\ref{helmholtz}) is determined by an eigenvector of the capacitance matrix $\mathcal C_N^{\gamma}$. Thus, the non-Hermitian skin effect is related to the exponential decay of the eigenvectors of the gauge capacitance matrix $\mathcal C_{N}^{\gamma}$.  Although a precise characterisation of the exponential decay of eigenvectors of a dense Toeplitz matrix remains an open problem, we can demonstrate the exponential decay of the pseudo-eigenvectors $\vect v$ of $\mathcal C_{N}^{\gamma}$ satisfying
\[
\frac{\bnorm{\left(\mathcal C_{N}^{\gamma}-\lambda I\right) \mathbf{v}}_2}{\bnorm{\vect v}_2} < \epsilon
\]
for a certain pseudo-eigenvalue $\lambda$ and small $\epsilon$. Here, $I$ denotes the identity matrix. Demonstrating the exponential decay of the pseudo-eigenvectors of $\mathcal C_{N}^{\gamma}$ is still non-trivial. A standard approach is to make use of a tridiagonal approximation similar to the nearest-neighbour approximation in quantum mechanics, where we only consider  interactions between neighbouring resonators. This results in a 
tridiagonal gauge capacitance matrix $\mathcal C_{N,2}^{\gamma}$ defined as in (\ref{equ:bandedcapamatrix1}) for $k=2$. The non-Hermitian skin effect for such model has been thoroughly discussed in \cite{ammari2023mathematical, dimerSkin}. However, due to the long-range interactions between the resonators, the elements of $\mathcal C_{N}^{\gamma}$ decay slowly as shown in (\ref{equ:decayofcapamatrixelements1}), making the nearest-neighbour approximation inaccurate. For instance, the first several modes of $\mathcal C_{N}^{\gamma}$ in Figure \ref{fig:k1N50} differ considerably from those of $\mathcal C_{N,2}^{\gamma}$. A straightforward generalisation of the nearest-neighbour approximation is a range-$k$ approximation, where we consider the interactions of neighbouring $k$-resonators. This results in a $k$-banded gauge capacitance matrix ${\mathcal C}_{N, k}^{\gamma}$ defined by (\ref{equ:bandedcapamatrix1}). In the next subsection, we shall characterise the exponential decay of the pseudo-eigenvectors of ${\mathcal C}_{N, k}^{\gamma}$ and in Subsection \ref{section:skineffectcapamatrix} we shall show that these pseudo-eigenvectors approximate well the pseudo-eigenvectors of the matrix $\mathcal C_{N}^{\gamma}$, given large enough $k$ and $N$. 

\subsection{Exponential decay of pseudo-eigenvectors of $k$-banded gauge capacitance matrices}\label{section:skineffectkbandedcapamatrix}
We first define a $k$-banded gauge capacitance matrix ${\mathcal C}_{N, k}^{\gamma}$ from $\mathcal C_N^{\gamma}$ as 
\begin{equation}\label{equ:bandedcapamatrix1}
\left({\mathcal C}_{N, k}^{\gamma}\right)_{i,j}= \left\{\begin{array}{ll}
\left(\mathcal C_N^{\gamma}\right)_{i,j}, & \text { for } |i-j|<k, \\
0, & \text { for } |i-j|\geq k,
\end{array} \quad 1 \leq i, j \leq N, \right.
\end{equation}
and demonstrate the exponential decay of its pseudo-eigenvectors. We let, for $\babs{i-j}<k$,  
\begin{equation}\label{equ:toeplitzelements1}
a_{j-i}=  \mathcal{C}_{i,j}^{\gamma}
\end{equation}
with $\mathcal{C}_{i,j}^{\gamma}$ being defined by (\ref{defrealspace}). The matrix whose $(i,j)$\textsuperscript{th} entry is $a_{i-j}$ is a Toeplitz matrix, whose symbol is given by
\begin{equation}\label{equ:bandedsymbol1}
f(z) = \sum_{j=-(k-1)}^{k-1} a_j z^j. 
\end{equation}
Define $\mathbb T:= \{z\in \mathbb C: \babs{z}=1\}$ and $\mathbb T_r: = \{z\in \mathbb C: \babs{z}=r\}$. Let $I(f(\mathbb T), \lambda)$ be the winding number of $f(\mathbb T)$ around $\lambda$ in the usual positive (counterclockwise) sense.



Define 
 \begin{equation}\label{equ:xiformula1}
 \xi(i, j) = \sum_{q=i}^{j}\frac{1}{q}.
 \end{equation}
Then by Abel's criteria, $\sum_{j=1}^{\infty}\frac{\xi(j+k_1,j+k_2)^2}{j^2}$ is convergent for any $k_1<k_2$ and $j+k_1\geq 1$. The following result on the pseudo-eigenvectors of $\mathcal C_{N,k}^{\gamma}$ holds. 
\begin{thm}\label{thm:pseduoeigenvectorthm0}
For any $\epsilon>0$, we can choose an integer $k>0$ so that 
$$
\max\left(\frac{1}{k}, \quad \sum_{j=k+1}^{+\infty}\frac{\xi(j+2-k, j+k)^2}{(1+j)^2}\right)\leq \epsilon
$$ 
with $\xi(\cdot)$ being defined by (\ref{equ:xiformula1}). Let $\lambda$ be any complex number with $I(f(\mathbb T), \lambda) \neq 0$ for $f$ defined by (\ref{equ:bandedsymbol1}). For some $\rho<1$ and any sufficiently large $N>4k$, there exist nonzero pseudo-eigenvectors $\mathbf{v}^{(N)}$ of $\mathcal C_{N,k}^{\gamma}$ with $\bnorm{\vect v^{(N)}}_2=1$ satisfying
\begin{equation}\label{equ:pseudovectorcapamatrixequ0}
\bnorm{\left(\mathcal C_{N,k}^{\gamma}-\lambda\right) \mathbf{v}^{(N)}}_2 < \max\left(C_1, C_2 N^{k-1}\right) \rho^{N-2k+2}+ \delta K \sqrt{10\epsilon}
\end{equation}
such that
\begin{equation}\label{equ:pseudovectorcapamatrixequ1}
\frac{\left|\left(\vect v^{(N)}\right)_j\right|}{\max_{j}\babs{\left(\vect v^{(N)}\right)_j}} \leq\left\{\begin{array}{ll}
C_3\rho^{j-1}, & \text { if } I(f(\mathbb T), \lambda)>0, \\
\nm
C_3\rho^{N-j}, & \text { if } I(f(\mathbb T), \lambda)<0,
\end{array} \quad 1 \leq j \leq N, \right.
\end{equation}
where $C_1, C_2, C_3$ are independent of $N$. In particular, the constant $\rho$ can be taken to be any number so that $1>\rho\geq r$ with $I\left(f\left(\mathbb T_r\right), \lambda\right)>0$ or $\frac{1}{\rho}\geq\frac{1}{r}>1$ with $I\left(f\left(\mathbb T_{\frac{1}{r}}\right), \lambda\right)<0$. 
\end{thm}
\begin{proof}
We decompose the $k$-banded matrix ${\mathcal C}_{N, k}^{\gamma}$ as 
\[
{\mathcal C}_{N,k}^{\gamma} = A+ M,
\]
where the entries of the matrices $A$ and $M$ are respectively given by
\begin{equation}
A_{i,j}=\begin{cases}
\left({\mathcal C}_{N,k}^{\gamma} \right)_{i,j}, & 1\leq i\leq k, \text{ or } N-(k-1)\leq i\leq N, \\
\nm
\mathcal C_{i, j}^{\gamma},  &\babs{i-j}< k, k<i< N-(k-1),\\
0, & \babs{i-j}\geq k,
\end{cases}
\end{equation}
and 
\begin{equation}
M_{i,j} = \begin{cases}
0, & 1\leq i\leq k, \text{ or } N-(k-1)\leq i\leq N,\\
\left({\mathcal C}_{N,k}^{\gamma} \right)_{i,j} - \mathcal C_{i,j}^\gamma,  &\babs{i-j}< k, k<i< N-(k-1),\\
0, & \babs{i-j}\geq k.
\end{cases}
\end{equation}
Note that, based on the definition of $\mathcal C_{i,j}^{\gamma}$ in (\ref{defrealspace}),  $A$ is a $k$-banded perturbed Toeplitz matrix like the matrix in (\ref{equ:perturbedtoeplitzmatrix1}). By Theorem \ref{thm:peudovectorperturbToeplitz1}, we have that for any   complex number $\lambda$ with $I(f(\mathbb T), \lambda) \neq 0$, where $f$ is defined by (\ref{equ:bandedsymbol1}) and for some $\rho<1$ and sufficiently large $N>4k$, there exist nonzero pseudo-eigenvectors $\mathbf{v}^{(N)}$ with $\bnorm{\vect v^{(N)}}_2=1$ satisfying
$$
\bnorm{\left(A-\lambda\right) \mathbf{v}^{(N)}}_2 \leq \max\left(C_1, C_2 N^{k-1}\right) \rho^{N-2k+2},
$$
such that
\begin{equation}\label{equ:proofbandedcapamatrix4}
\frac{\left|\left(\vect v^{(N)}\right)_j\right|}{\max_{j}\babs{\left(\vect v^{(N)}\right)_j}} \leq\left\{\begin{array}{ll}
C_3\rho^{j-1}, & \text { if } I(f(\mathbb T), \lambda)>0, \\
C_3\rho^{N-j}, & \text { if } I(f(\mathbb T), \lambda)<0,
\end{array} \quad 1 \leq j \leq N, \right.
\end{equation}
where $C_1, C_2, C_3$ are independent of $N$. To prove (\ref{equ:pseudovectorcapamatrixequ0}), we  estimate $\bnorm{M\vect v^{(N)}}_2$. By the definition of $M$, we have $\left(M \vect v^{(N)}\right)_{j}=0$ for $1\leq j\leq k, N-(k-1)\leq j\leq N$. For $k+1\leq j\leq \lfloor \frac{N}{2}\rfloor-k$, we have  
\begin{align*}
\babs{\left(M \vect v^{(N)}\right)_{j}} =& \babs{\sum_{q=j+1-k}^{j+k-1}M_{j, q}\left(\vect v^{(N)}\right)_q} \leq \sum_{q=j+1-k}^{j+k-1}\babs{M_{j, q}}\babs{\left(\vect v^{(N)}\right)_q} \\
\leq & \sum_{q=j+1-k}^{j+k-1} \frac{\delta K}{(1+j)(1+q)}\quad \left(\text{by Lemma \ref{lem:capacitancetotoeplitz2}, $\bnorm{\vect v^{(N)}}_2=1$, and $N>4k$}\right)\\
= & \frac{1}{1+j}\sum_{q=j+1-k}^{j+k-1} \frac{\delta K}{1+q}= \frac{\delta K}{1+j} \xi(j+2-k, j+k), 
\end{align*}
where $\xi(j+2-k, j+k)$ is defined as in (\ref{equ:xiformula1}). Therefore, 
\[
\sum_{j=k+1}^{\lfloor \frac{N}{2}\rfloor-k}\babs{\left(M \vect v^{(N)}\right)_{j}}^2 \leq (\delta K)^2\sum_{j=k+1}^{\lfloor \frac{N}{2}\rfloor}\frac{\xi(j+2-k, j+k)^2}{(1+j)^2}<(\delta K)^2 \epsilon,
\]
where the last inequality is from the condition on $k$. 
Similarly, we can also prove that 
\[
\sum_{j=\lfloor \frac{N}{2}\rfloor+k}^{N-k}\babs{\left(M \vect v^{(N)}\right)_{j}}^2< (\delta K)^2 \epsilon.
\]
For $\lfloor \frac{N}{2}\rfloor-k+1\leq j\leq \lfloor \frac{N}{2}\rfloor+k-1$, we have 
\begin{align*}
\babs{\left(M \vect v^{(N)}\right)_{j}} =& \babs{\sum_{q=j+1-k}^{j+k-1}M_{j, q}\left(\vect v^{(N)}\right)_q} \leq \sum_{q=j+1-k}^{j+k-1}\babs{M_{j, q}}\babs{\left(\vect v^{(N)}\right)_q} \\
\leq & \sum_{q=j+1-k}^{j+k-1} \frac{\delta K}{(1+\lfloor\frac{N}{2}\rfloor-k+1)(1+\lfloor\frac{N}{2}\rfloor-k+1)}\ \left(\text{by Lemma \ref{lem:capacitancetotoeplitz2}, $\bnorm{\vect v^{(N)}}_2=1, N>4k$}\right)\\
\leq & \frac{2\delta K}{k} \quad \left(\text{since $N> 4k$}\right).
\end{align*}
It follows that 
\[
\sum_{j = \lfloor \frac{N}{2}\rfloor-k+1}^{\lfloor \frac{N}{2}\rfloor+k-1} \babs{\left(M \vect v^{(N)}\right)_{j}}^2\leq \frac{8(\delta K)^2}{k}\leq 8(\delta K)^2\epsilon. 
\]
Combining all above estimates, we have   
\[
\bnorm{M \vect v^{(N)}}_2 <  \delta K \sqrt{10\epsilon},
\]
and 
\begin{align*}
\bnorm{\left(\mathcal C_{N,k}^{\gamma}-\lambda\right) \mathbf{v}^{(N)}}_2\leq& \bnorm{\left(A-\lambda\right) \mathbf{v}^{(N)}}_2 + \bnorm{M\vect v^{(N)}}_2 \\
< &\max\left(C_1, C_2 N^{k-1}\right) \rho^{N-2k+2}+\delta K \sqrt{10\epsilon}.
\end{align*}
\end{proof}

We have demonstrated exponential decay of the pseudo-eigenvectors of the $k$-banded matrix $\mathcal C_{N,k}^\gamma$ for sufficiently large $N$. In particular, by Theorem \ref{thm:pseduoeigenvectorthm0}, when the winding number $I(f(\mathbb T), \lambda)$ for the corresponding symbol $f$ is non-zero, then there must be pseudo-eigenvectors of $\lambda$ with certain exponential decay. On the other hand, we stress that the exponential decay property of the pseudo-eigenvectors may not be valid for general Toeplitz matrices; see \cite{trefethen2005spectra}. 

Finally, we numerically illustrate the results in  Theorem \ref{thm:pseduoeigenvectorthm0}. In particular, Figures \ref{fig:symbol1} and \ref{fig:symbolm1} show respectively the symbol functions of the $k$-banded Toeplitz matrices $\mathcal{C}_{N,k}^\gamma$ for $\gamma=\pm 1$ evaluated on the unit circle $\mathbb{T}$. As predicted by Theorem \ref{thm:pseduoeigenvectorthm0}, all the points inside the circle are the pseudo-eigenvalues of $\mathcal C_{N,k}^{\gamma}$ for large enough $N$, including the black dots, which are the eigenvalues of the gauge capacitance matrix $\mathcal C_{N}^{\gamma}$. \Cref{fig:symbol1b} shows the eigenvectors of $\mathcal C_{N}^{\gamma}$ corresponding to the black and red dots. In particular, the localised eigenmodes in gray are the pseudo-eigenmodes of $\mathcal{C}_{N,k}^\gamma$. As seen from \Cref{fig:symbol1,fig:symbolm1}, the winding numbers and the positions of the pseudo-eigenvalues correctly predict the exponential decay of the corresponding pseudo-eigenmodes. As seen in red in \Cref{fig:symbol1b,fig:symbolm1b}, the non-localised mode corresponds to the eigenvalue of $\mathcal C_N^{\gamma}$ outside the enclosed region with non-zero winding.

    \begin{figure}[!h]
        \centering
        \begin{subfigure}[b]{0.45\textwidth}
            \includegraphics[width=\textwidth]{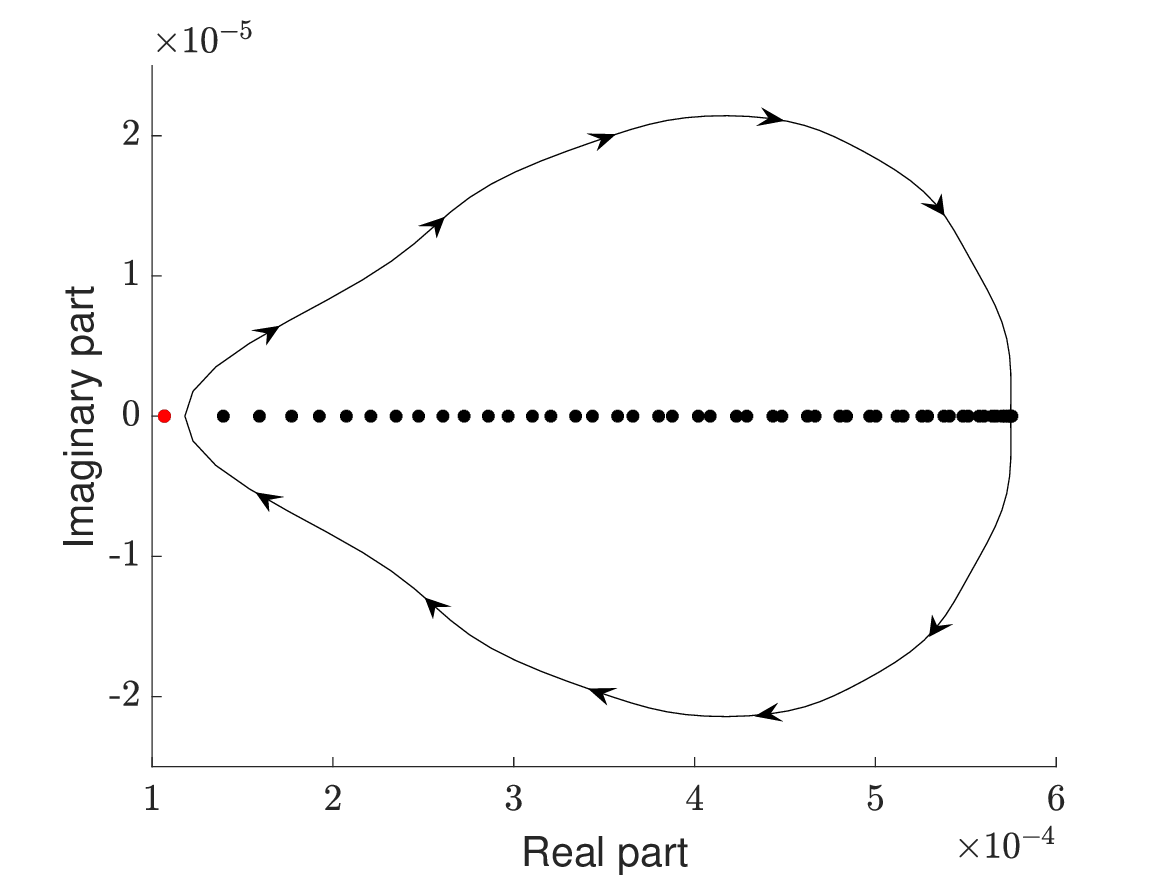}
            \caption{Symbol function $f(\mathbb{T})$ and eigenvalues (asterisks) plotted in the complex plane.}
        \end{subfigure}        
        \begin{subfigure}[b]{0.45\textwidth}
            \includegraphics[width=\textwidth]{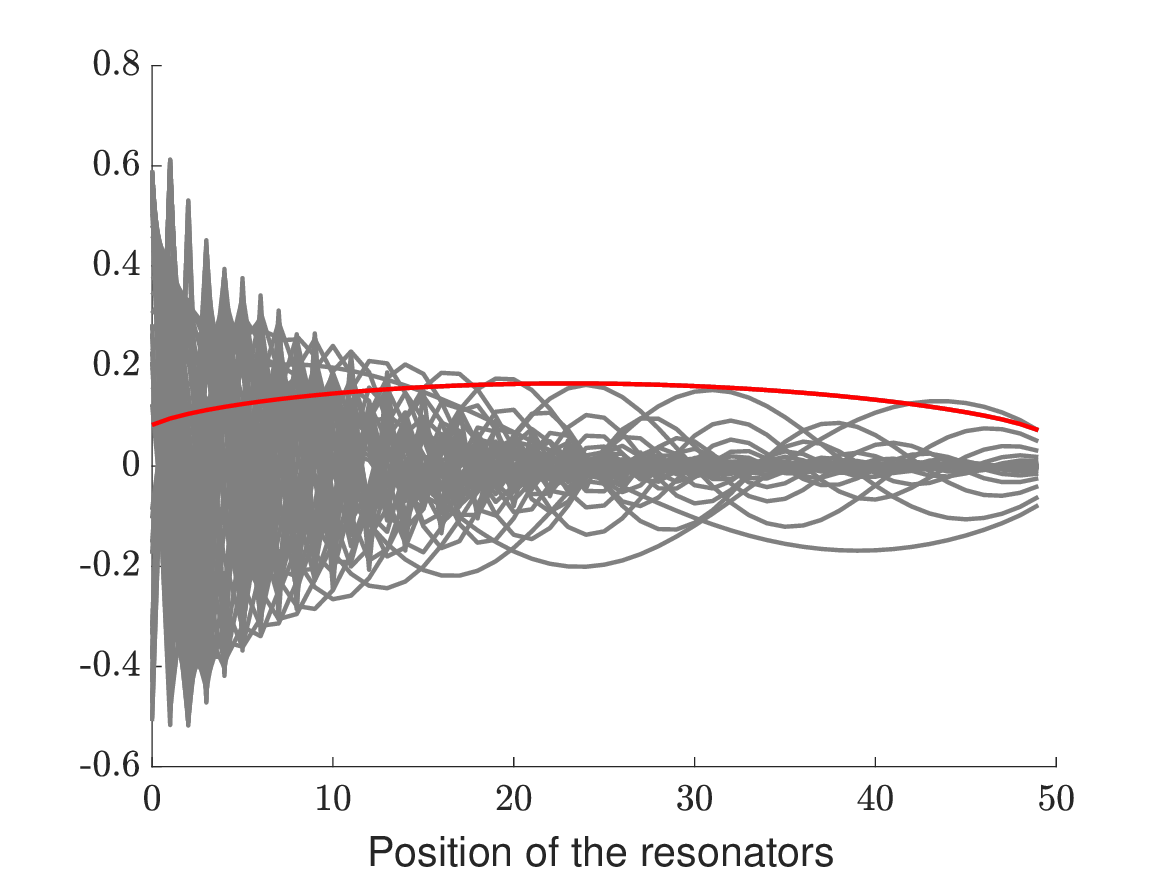}
            \caption{Eigenmodes of $\mathcal{C}_{N}^\gamma$. \\ \hspace{0pt} }
            \label{fig:symbol1b}
        \end{subfigure}     
        \caption{Chain of resonators with $N=50$ and $\gamma = 1$; symbol function of $\mathcal{C}_{N,k}^\gamma$ showing  a winding number equals to $-1$. Here, we choose $k=10$.  \label{fig:symbol1}}
    \end{figure}
        \begin{figure}[!h]
        \centering
        \begin{subfigure}[b]{0.45\textwidth}
            \includegraphics[width=\textwidth]{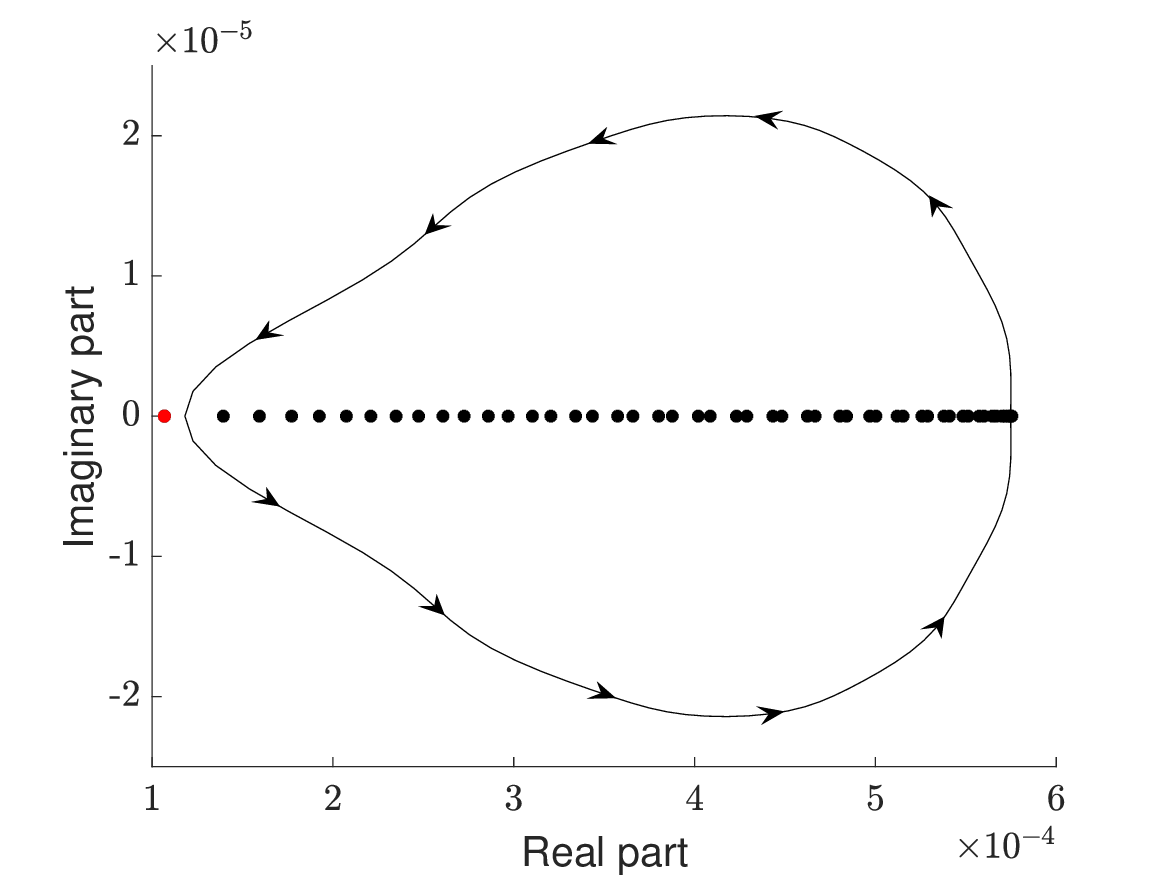}
            \caption{Symbol function $f(\mathbb{T})$ and eigenvalues (asterisks) plotted in the complex plane. }
        \end{subfigure}        
        \begin{subfigure}[b]{0.45\textwidth}
            \includegraphics[width=\textwidth]{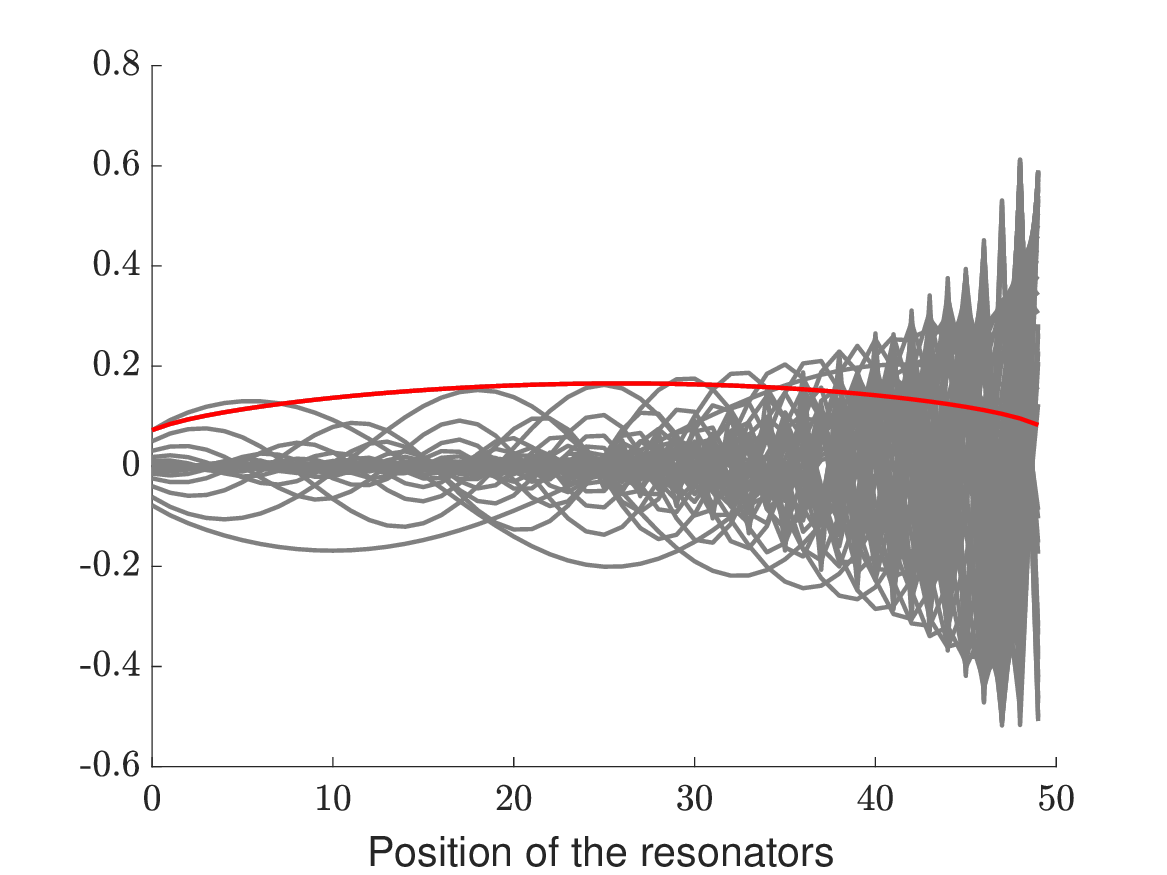}
            \caption{Eigenmodes of $\mathcal{C}_{N}^\gamma$.\\ \hspace{0pt} \label{fig:symbolm1b} }
        \end{subfigure}     
        \caption{Chain of resonators with $N=50$ and $\gamma = -1$; symbol function of $\mathcal{C}_{N,k}^\gamma$ showing  a winding number equals to $1$. Here, we choose $k=10$.}
        \label{fig:symbolm1}
    \end{figure}



\subsection{Exponential decay of pseudo-eigenvectors of gauge capacitance matrices}\label{section:skineffectcapamatrix}
In Subsection \ref{section:skineffectkbandedcapamatrix}, we have demonstrated the exponential decay of the pseudo-eigenvectors of the $k$-banded gauge capacitance matrix $\mathcal C_{N,k}^{\gamma}$ corresponding to $\lambda$ so that $I(f(\mathbb T), \lambda)\neq 0$ and numerically illustrated it. In this section, we demonstrate that, for sufficiently large $k$, the pseudo-eigenvectors of $\mathcal C_{N,k}^{\gamma}$ having the exponential decay property are also  pseudo-eigenvectors of $\mathcal C_N^{\gamma}$. In particular, we have the following theorem. 
\begin{thm}\label{thm:pseduoeigenvectorthm1}
For any $\epsilon_1>0$, we can choose an integer $k>0$ so that 
\[
\max\left(\frac{1}{k},\ \sum_{j=k}^{+\infty}\frac{1}{j^2}\right)< \epsilon_1.
\]
Then, for any unit pseudo-eigenvector of $ {\mathcal C}_{N,k}^{\gamma}$ satisfying 
\begin{equation}\label{equ:pseduoeigenvectorthm1equ0}
{\bnorm{\left({\mathcal {C}}_{N,k}^{\gamma}-\lambda I\right) \mathbf{v}^{(N)}}_2} \leq \epsilon_2, 
\end{equation}
and
\begin{equation}\label{equ:pseduoeigenvectorthm1equ1}
\frac{\left|\left(\vect v^{(N)}\right)_j\right|}{\max_{j}\babs{\left(\vect v^{(N)}\right)_j}}\leq 
C\rho^{j-1}, j=1,\ldots, N, \text{ or }\ \frac{\left|\left(\vect v^{(N)}\right)_j\right|}{\max_{j}\babs{\left(\vect v^{(N)}\right)_j}}\leq 
C\rho^{N-j}, j=1,\ldots, N,
\end{equation}
for some constant $C$ and $0<\rho<1-\epsilon_1$, we have
\begin{equation}\label{equ:pseudoeigenvectorequ1}
{\bnorm{\left(\mathcal {C}_N^{\gamma}-\lambda I\right) \mathbf{v}^{(N)}}_2}\leq \epsilon_2+\delta K C\left(\frac{1-\epsilon_1}{1-\epsilon_1-\rho}\sqrt{\epsilon_1}+\epsilon_1 \rho^{k}\frac{1}{1-\rho}\sqrt{\frac{1}{1-\rho^2}}\right).
\end{equation}
\end{thm}
\begin{proof}
By Proposition \ref{lem:capacitancematrixesitmate1}, we have 
\begin{equation}\label{equ:proofpseduoeigenvectorthm1}
\abs{(\mathcal{C}^\gamma_N)_{i, j}} \leq  \frac{\delta K}{\abs{i-j}},
\end{equation}
for some constant $K$.  
We decompose $\mathcal C_{N}^\gamma$ as follows
\[
\mathcal C_{N}^\gamma = {\mathcal C}_{N,k}^{\gamma} +L+R,
\]
where the matrices $L$ and $R$ are respectively defined  by
\[
\left(L\right)_{ij}=\begin{cases}
\left(\mathcal C_{N}^{\gamma} \right)_{i, j}, & i-j\geq  k, \\
0, & i-j< k,
\end{cases}
\]
and 
\[
\left(R\right)_{ij}=\begin{cases}
\left(\mathcal C_{N}^{\gamma} \right)_{i, j}, & i-j\leq -k, \\
0, & i-j> -k.
\end{cases}
\]
In order to prove (\ref{equ:pseudoeigenvectorequ1}), we estimate $\bnorm{L \vect v^{(N)}}_2$ and $\bnorm{R \vect v^{(N)}}_2$. Since $\bnorm{\vect v^{(N)}}_2=1$  implies $\max_{j}\left|\left(\vect v^{(N)}\right)_j\right|\leq 1$, (\ref{equ:pseduoeigenvectorthm1equ1}) gives
\begin{equation}\label{equ:proofpseduoeigenvectorthm3}
\left|\left(\vect v^{(N)}\right)_j\right|\leq C \rho^{j-1},  1 \leq j \leq N,\ { or } \left|\left(\vect v^{(N)}\right)_j\right|\leq C \rho^{N-j},  1 \leq j \leq N. 
\end{equation}
First, for $\vect v^{N}$ satisfying 
\begin{equation}\label{equ:proofpseduoeigenvectorthm4}
\left|\left(\vect v^{(N)}\right)_j\right|\leq C\rho^{j-1}, \quad 1\leq j\leq N,
\end{equation}
we calculate  $\bnorm{L \vect v^{(N)}}_2$. Note that the entries of $L\vect v^{(N)}$ are given by
\begin{equation}\label{equ:proofpseduoeigenvectorthm2}
(L\vect v^{(N)})_{j} =\begin{cases}0, & j=1,\ldots, k, \\
\nm
 \ds \sum_{q=k+1}^{j} \left(\mathcal C_N^{\gamma}\right)_{j, q-k} \left(\vect v^{(N)}\right)_{q-k}, & j=k+1,\ldots, N. 
\end{cases}
\end{equation}
Define 
\begin{equation}\label{equ:proofpseduoeigenvectorthm6}
s: = \min_{p\geq k}\min_{q=1,2,\cdots}\left(\frac{p}{p+q}\right)^{\frac{1}{q}}. 
\end{equation}
It is not hard to see that 
\begin{equation}\label{equ:proofpseduoeigenvectorthm7}
1-\epsilon_1 <\frac{k}{k+1}\leq s \leq 1. 
\end{equation}
By (\ref{equ:proofpseduoeigenvectorthm1}) and (\ref{equ:proofpseduoeigenvectorthm4}), we have for $j=k+1, \ldots, N$,
\begin{align*}
\babs{(L\vect v^{(N)})_{j}} = & \babs{\sum_{q=k+1}^{j} \left(\mathcal C_N^{\gamma}\right)_{j, q-k}  \left(\vect v^{(N)}\right)_{q-k}}\leq \sum_{q=k+1}^{j} \babs{\left(\mathcal C_N^{\gamma}\right)_{j, q-k} } \babs{\left(\vect v^{(N)}\right)_{q-k}} \\
\leq  & \delta KC\sum_{q=k+1}^{j} \frac{1}{j-q+k}\rho^{q-k-1} = \delta KC\sum_{q=k+1}^{j} \frac{1}{q-1}\rho^{j-q} \\
=&\delta KC\sum_{q=k+1}^{j} \frac{1}{q-1}s^{j-q} (\rho/s)^{j-q}\\
= & \frac{\delta KC}{j-1}\left(\sum_{q=k+1}^{j} \frac{j-1}{q-1}s^{j-q} (\rho/s)^{j-q} \right) \\
= & \frac{\delta KC}{j-1}\left(\sum_{q=k+1}^{j} \frac{q-1+j-q}{q-1}s^{j-q} (\rho/s)^{j-q} \right)\\
\leq & \frac{\delta KC}{j-1}\left(\sum_{q=k+1}^{j} (\rho/s)^{j-q} \right)\quad \left(\text{by (\ref{equ:proofpseduoeigenvectorthm6})}\right)\\
\leq& \frac{\delta KC}{j-1}\sum_{q=0}^{\infty}\left(\frac{\rho}{1-\epsilon_1}\right)^{q} \quad \left(\text{by (\ref{equ:proofpseduoeigenvectorthm7})}\right)\\
= & \frac{\delta KC}{j-1} \frac{1-\epsilon_1}{1-\epsilon_1-\rho}. 
\end{align*}
Thus, combining the above estimates, we obtain that
\begin{align}\label{equ:proofpseduoeigenvectorthm8}
\bnorm{L\vect v^{(N)}}_2\leq \delta KC\frac{1-\epsilon_1}{1-\epsilon_1-\rho}\sqrt{\sum_{j=k+1}^N\frac{1}{(j-1)^2}}\leq \delta KC\frac{1-\epsilon_1}{1-\epsilon_1-\rho}\sqrt{\epsilon_1},
\end{align}
where the last inequality is from the condition on $k$ in the theorem.

Next, for $\vect v^{(N)}$ satisfying 
\begin{equation}\label{equ:proofpseduoeigenvectorthm5}
\left|\left(\vect v^{(N)}\right)_j\right|\leq C \rho^{N-j}, \quad 1\leq j\leq N,
\end{equation}
we estimate $\bnorm{L \vect v^{(N)}}_2$. The entries of $L\vect v^{(N)}$ are given by (\ref{equ:proofpseduoeigenvectorthm2}). By (\ref{equ:proofpseduoeigenvectorthm1}) and (\ref{equ:proofpseduoeigenvectorthm5}),  we similarly have, for $j=k+1,\ldots, N$, that
\begin{align*}
\babs{(L\vect v^{(N)})_{j}} = & \babs{\sum_{q=k+1}^{j} \left(\mathcal C_N^{\gamma}\right)_{j, q-k}  \left(\vect v^{(N)}\right)_{q-k}}\leq \sum_{q=k+1}^{j} \babs{\left(\mathcal C_N^{\gamma}\right)_{j, q-k} } \babs{\left(\vect v^{(N)}\right)_{q-k}} \\
\leq  & \delta KC\sum_{q=k+1}^{j} \frac{1}{j-q+k}\rho^{N-(q-k)} = \delta KC\sum_{q=k+1}^{j} \frac{1}{q-1}\rho^{N-1-(j-q)}.
\end{align*}
Further estimating, we have
\begin{align*}
 \babs{(L\vect v^{(N)})_{j}} &\leq  \delta KC\sum_{q=k+1}^{j} \frac{1}{q-1}\rho^{N-1-(j-q)}\\
&\leq  \frac{\delta KC}{k}\rho^{N-1-j+(k+1)}\left(\sum_{q=k+1}^{j}\frac{k}{q-1}\rho^{q-(k+1)}\right)\\
&\leq  \frac{\delta KC}{k}\rho^{N-1-j+(k+1)} \left(\sum_{q=k+1}^{j}\rho^{q-(k+1)}\right)\\
&\leq  \frac{\delta KC}{k}\rho^{N-1-j+(k+1)} \frac{1}{1-\rho}.
\end{align*}
Therefore, combining the above estimates yields
\begin{equation}
\begin{aligned}
\bnorm{L\vect v^{(N)}}_2\leq& \frac{\delta KC}{k}\frac{1}{1-\rho}\sqrt{\sum_{j=k+1}^N\rho^{2(N-1-j+(k+1))}}\\
\leq& \frac{\delta KC}{k}\rho^{k}\frac{1}{1-\rho}\sqrt{\frac{1}{1-\rho^2}}\\
\leq& \delta KC\epsilon_1\rho^{k}\frac{1}{1-\rho}\sqrt{\frac{1}{1-\rho^2}},
\end{aligned}
\end{equation}
where the last inequality is from the condition on $k$ in the theorem.

In the same fashion, we can prove that for $\vect v^{(N)}$ satisfying (\ref{equ:proofpseduoeigenvectorthm4}), we have  
\[
\bnorm{R\vect v^{(N)}}_2\leq \delta KC\epsilon_1 \rho^{k}\frac{1}{1-\rho}\sqrt{\frac{1}{1-\rho^2}},
\]
and for 
$\vect v^{(N)}$ satisfying (\ref{equ:proofpseduoeigenvectorthm5}), we have  
\[
\bnorm{R\vect v^{(N)}}_2\leq \delta K C\frac{1-\epsilon_1}{1-\epsilon_1-\rho}\sqrt{\epsilon_1}.
\]
Therefore, we arrive at 
\begin{align*}
\bnorm{(\mathcal {C}_N^{\gamma}-\lambda I) \mathbf{v}^{(N)}}_2\leq& \bnorm{\left({\mathcal C}_{N,k}^{\gamma}-\lambda I\right) \mathbf{v}^{(N)}}_2+\bnorm{L\mathbf{v}^{(N)}}_2+  \bnorm{R\mathbf{v}^{(N)}}_2\\
\leq & \epsilon_2+\delta K C\left(\frac{1-\epsilon_1}{1-\epsilon_1-\rho}\sqrt{\epsilon_1}+\epsilon_1 \rho^{k}\frac{1}{1-\rho}\sqrt{\frac{1}{1-\rho^2}}\right),
\end{align*}
which completes the proof. 
\end{proof}

Theorem \ref{thm:pseduoeigenvectorthm1} elucidates that, for sufficiently large $k>0$, an exponentially decaying pseudo-eigenvector of the $k$-banded  matrix $\mathcal C_{N,k}^{\gamma}$ is also a pseudo-eigenvector of the gauge capacitance matrix $\mathcal C_N^{\gamma}$. Theorems \ref{thm:pseduoeigenvectorthm0} and \ref{thm:pseduoeigenvectorthm1} indicate that the pseudo-eigenvectors/eigenvectors of the gauge capacitance matrix are exponentially decaying when  $I(f(\mathbb T),\lambda)$ is not zero with $f$ being the symbol associated with a sufficiently large $k$-banded submatrix. In particular, together with the numerical results in Figures \ref{fig:symbol1} and \ref{fig:symbolm1}, it is shown that most of the eigenmodes of $\mathcal C_N^{\gamma}$ have the exponential decay property, as most of the eigenvalues of $\mathcal C_N^{\gamma}$ lies inside $f(\mathbb T)$. This verifies the non-Hermitian skin effect of the resonating system (\ref{helmholtz}). 



    Figure \ref{fig:k10N50} shows the first $20$ eigenvectors $\vect v_j$'s of the gauge capacitance matrix $\mathcal{C}_N^\gamma$ (black line) and the corresponding $20$ pseudo-eigenvectors $\vect v_j(\epsilon)$'s of the $k$-banded matrix $\mathcal{C}_{N,k}^\gamma$ for $k=10$ (blue line). We observe that, as predicted by Theorem \ref{thm:pseduoeigenvectorthm1}, pseudo-eigenvectors (and hence, eigenvectors) with stronger exponential decay can be better approximated by corresponding pseudo-eigenvectors of the $k$-banded matrix.
   \begin{figure}[!h]
        \centering
            \includegraphics[width= 0.95\textwidth]{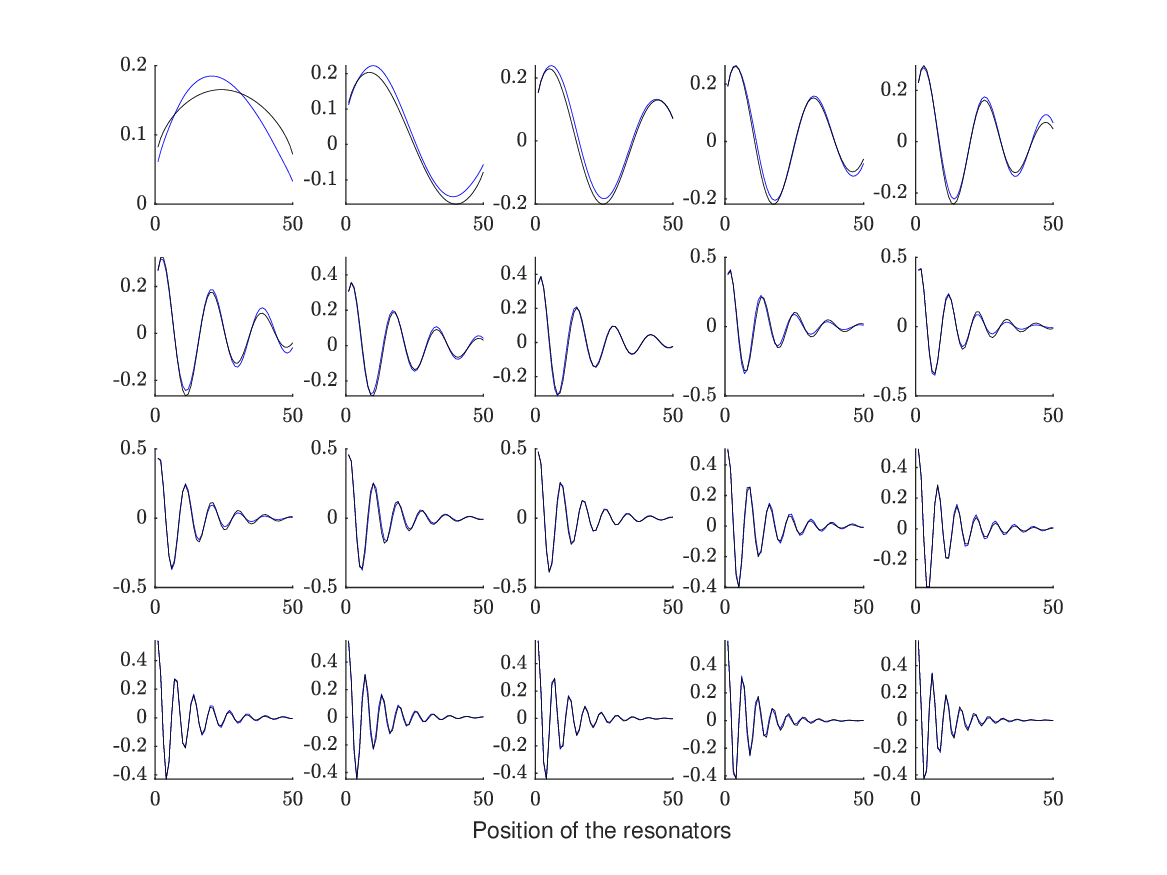}
            \caption{The first $20$ eigenmodes of the gauge capacitance matrix $\mathcal{C}_{N}^\gamma$ are in black. And the blue modes are the pseudo-eigenmodes of the $10$-banded matrix $\mathcal{C}_{N,k}^\gamma$. The eigenmodes are normalized.}
        \label{fig:k10N50}
    \end{figure}

    In Figure \ref{fig:k1N50}, we compare the eigenvectors of the tridiagonal capacitance matrix $\mathcal C_{N,2}^{\gamma}$  and those of $\mathcal C_{N}^{\gamma}$. It is first shown that although the nearest-neighbour approximation is insufficient to approximate all the exponentially decaying eigenvectors of $\mathcal C_{N}^{\gamma}$, it does approximate the ones decaying fast enough. Secondly, we observe that, for the first several modes, due to long-range interactions, the non-Hermitian skin effect in three-dimensional systems of resonators is less pronounced than that predicted by the nearest-neighbour approximation. Thus long-range interactions mainly affect the first several eigenmodes in systems of subwavelength resonators.  
        \begin{figure}[!h]
        \centering
            \includegraphics[width=0.95\textwidth]{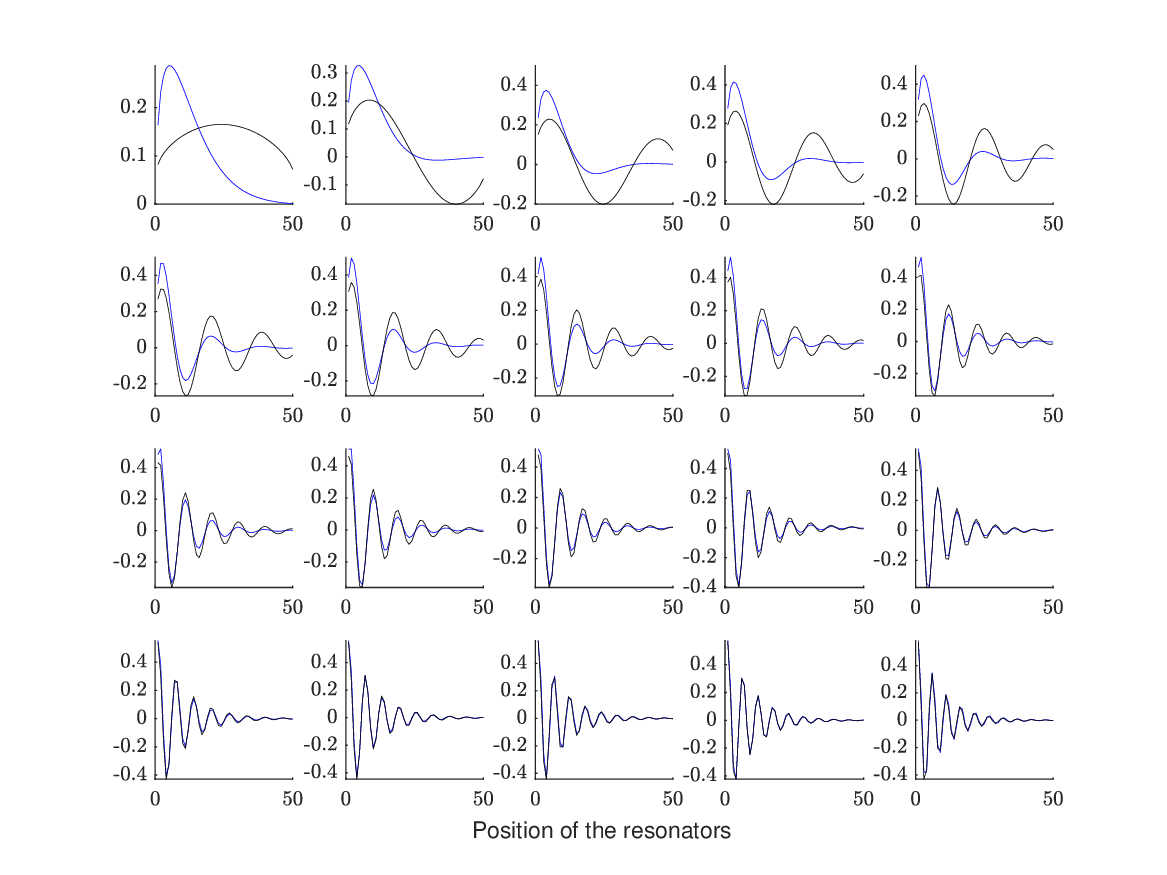}
            \caption{The first $20$ modes of the tridiagonal part of the matrix (in blue) and the full matrix (in black). }
        \label{fig:k1N50}
    \end{figure}

\section{Numerical illustrations of the non-Hermitian skin effect} \label{sect7}
In this section, we provide a variety of further numerical illustrations of the skin effect. We begin by studying the stabiltiy of the skin effect in the presence of disorder. We also compute the skin effect in two-dimensional lattice structures.

\subsection{Numerical simulations of the non-Hermitian skin effect in chains of subwavelength resonators}
   
    In Figure \ref{fig:chain_gammas}, we consider a chain of $100$ identical equidistant spherical resonators in one line aligned with the $x_1$-axis. We plot all the eigenmodes in one graph and demonstrate the condensation of the eigenvectors at one edge of the chain. The condensation becomes more pronounced with increasing $\gamma$. Furthermore, changing the sign of $\gamma$ changes the direction of condensation.
\begin{figure}[!h]
\centering
    \begin{subfigure}[h]{0.33\textwidth}
    \includegraphics[width=\textwidth]{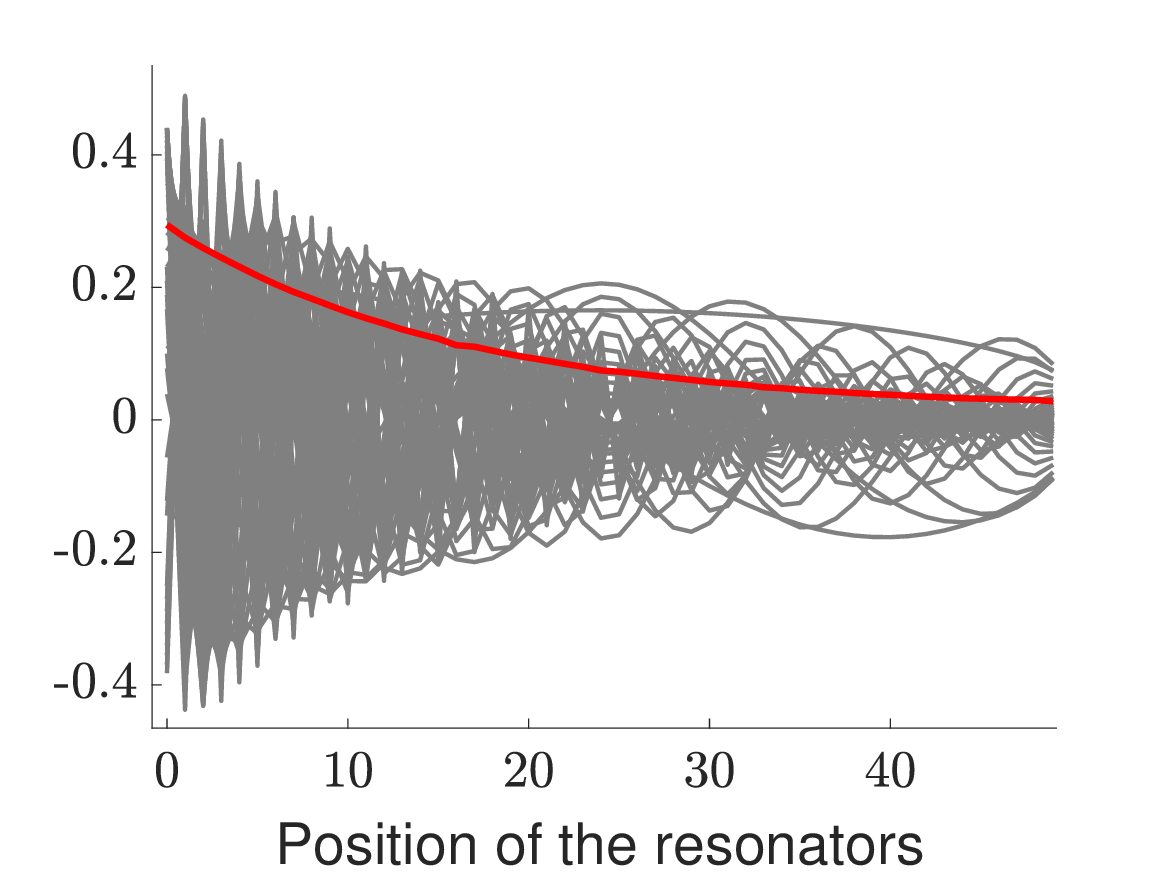}
    \caption{Eigenmodes for $\gamma = 0.5$.}
    \end{subfigure}\hfill
    \begin{subfigure}[h]{0.33\textwidth}
    \includegraphics[width=\textwidth]{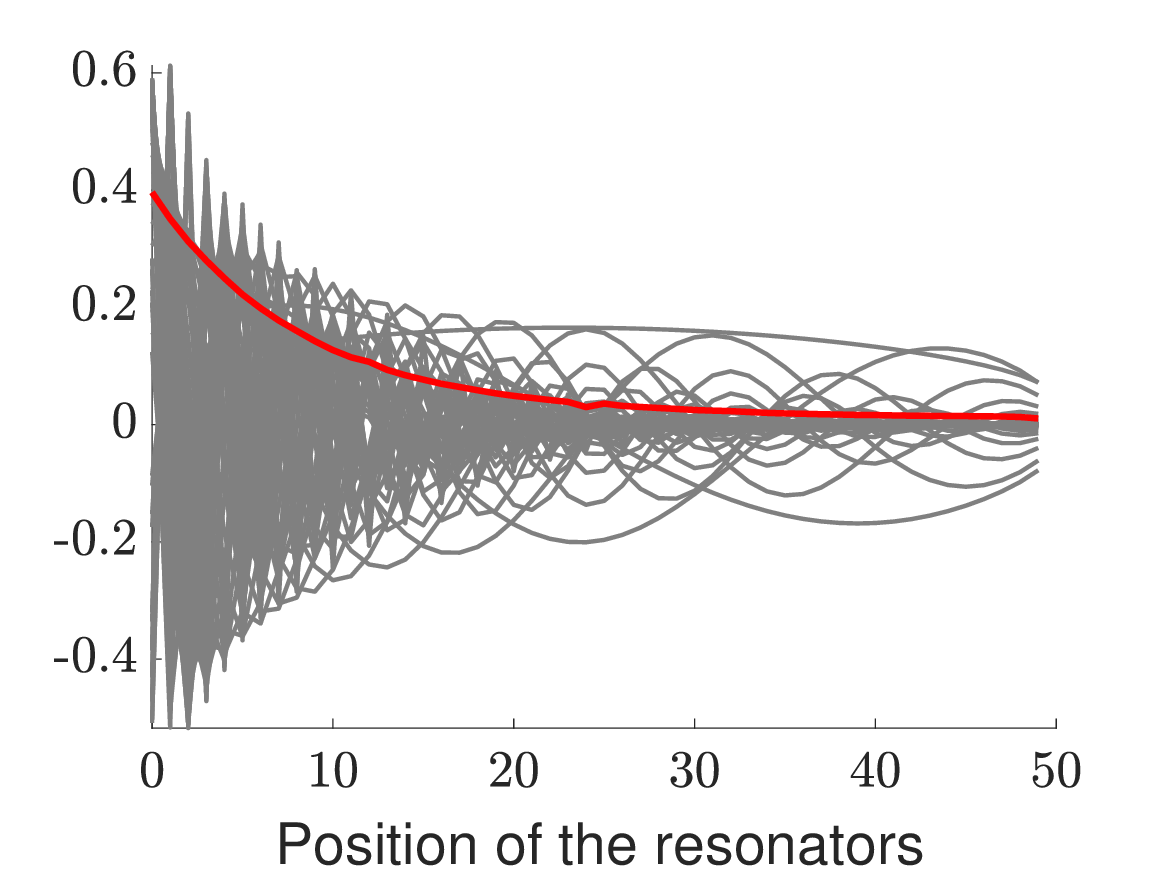}    
    \caption{Eigenmodes for $\gamma = 1$.}
    \end{subfigure}\hfill
    \begin{subfigure}[h]{0.33\textwidth}
    \includegraphics[width=\textwidth]{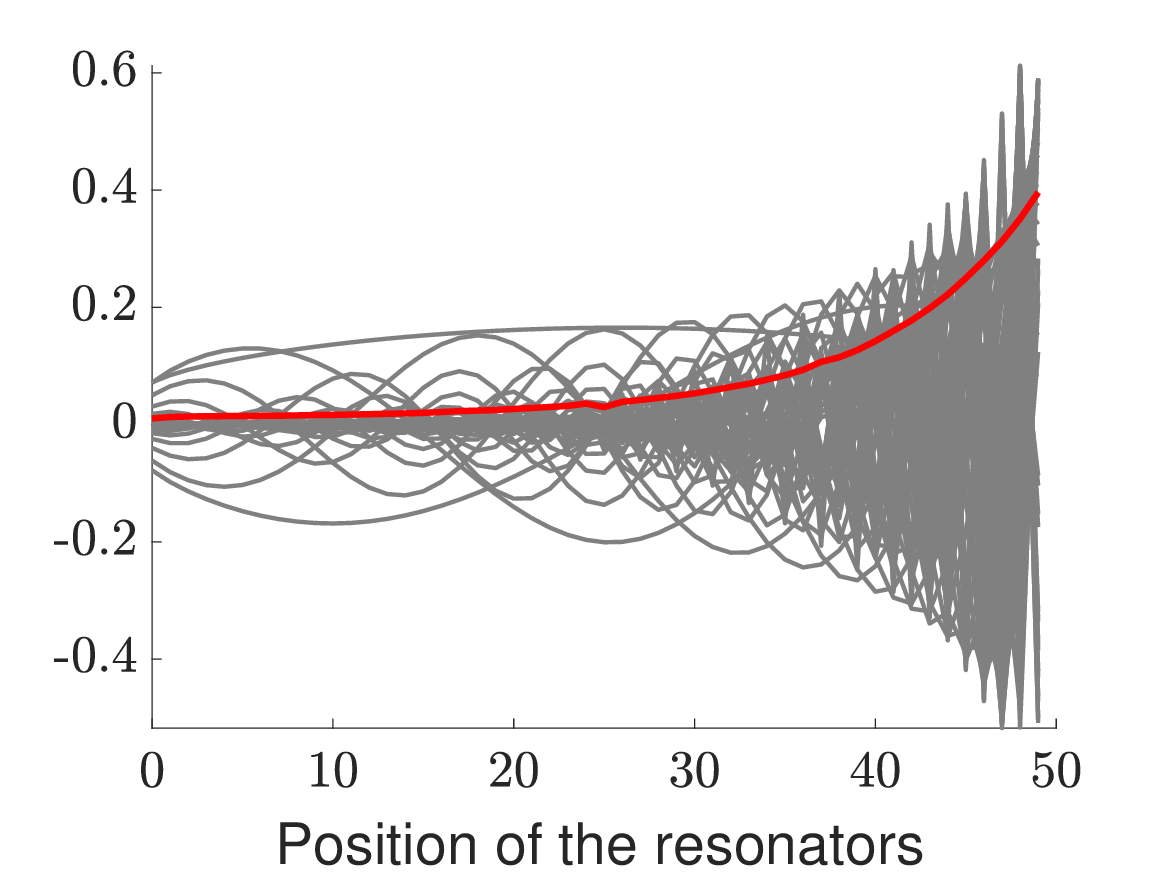}
    \caption{Eigenmodes for $\gamma = -1$.}
    \end{subfigure}
    \caption{We plot all the eigenmodes for different values of $\gamma$. The red line symbolises the average of the absolute value of the amplitudes. \label{fig:chain_gammas}}
\end{figure}

    We say that an eigenmode is condensated if the ratio between the norm of its restriction to the first $20\%$ resonators and its entire norm is greater than $80\%$. We then count the number of condensated eigenmodes and compute the ratio to $N$ (the total number of eigenmodes). In Figure \ref{fig:condensated}, we plot the portion of condensanted eigenmodes when increasing $\gamma$ or $N$.
    \begin{figure}[!h]
    \centering
    \begin{subfigure}[h]{0.48\textwidth}
    \includegraphics[width=\textwidth]{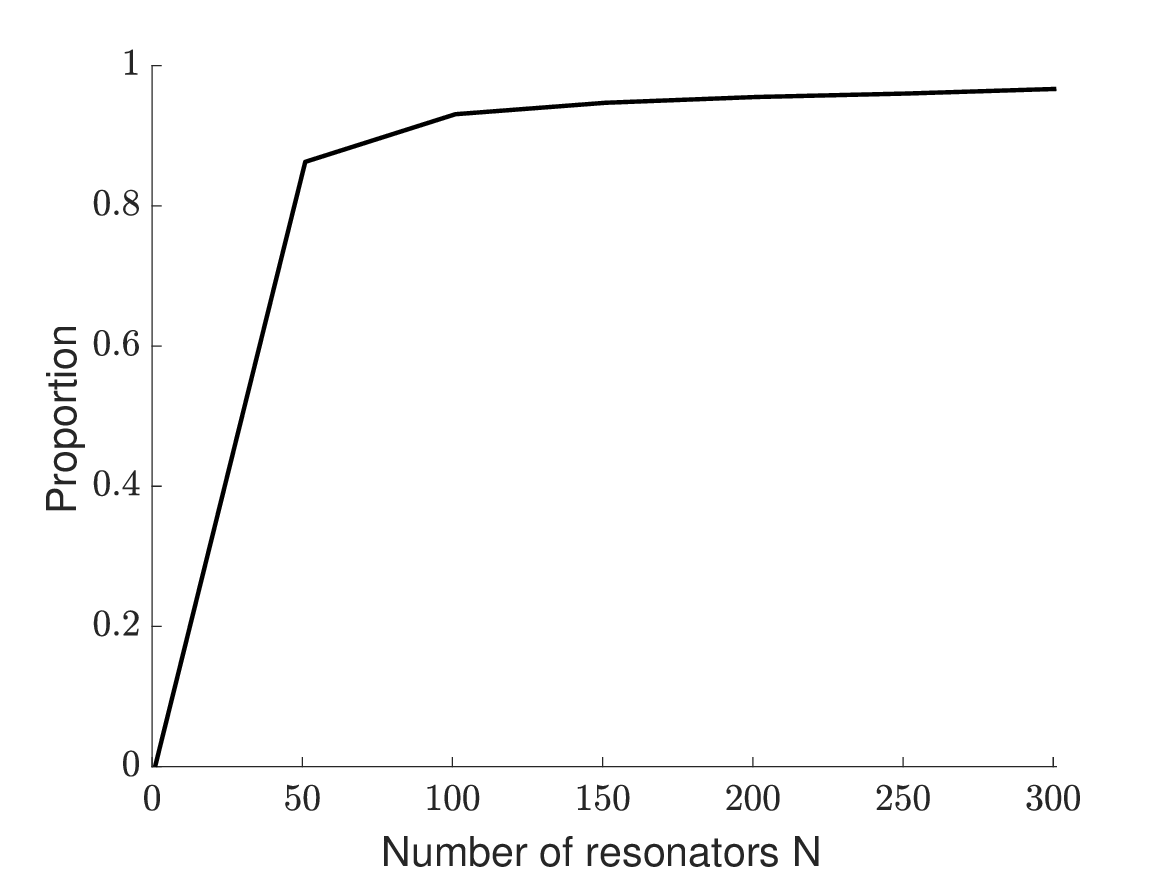}
    \caption{Proportion of condensated eigenmodes with $\gamma=1$ and 
    increasing $N$.}
    \end{subfigure}
    \begin{subfigure}[h]{0.48\textwidth}
    \includegraphics[width=\textwidth]{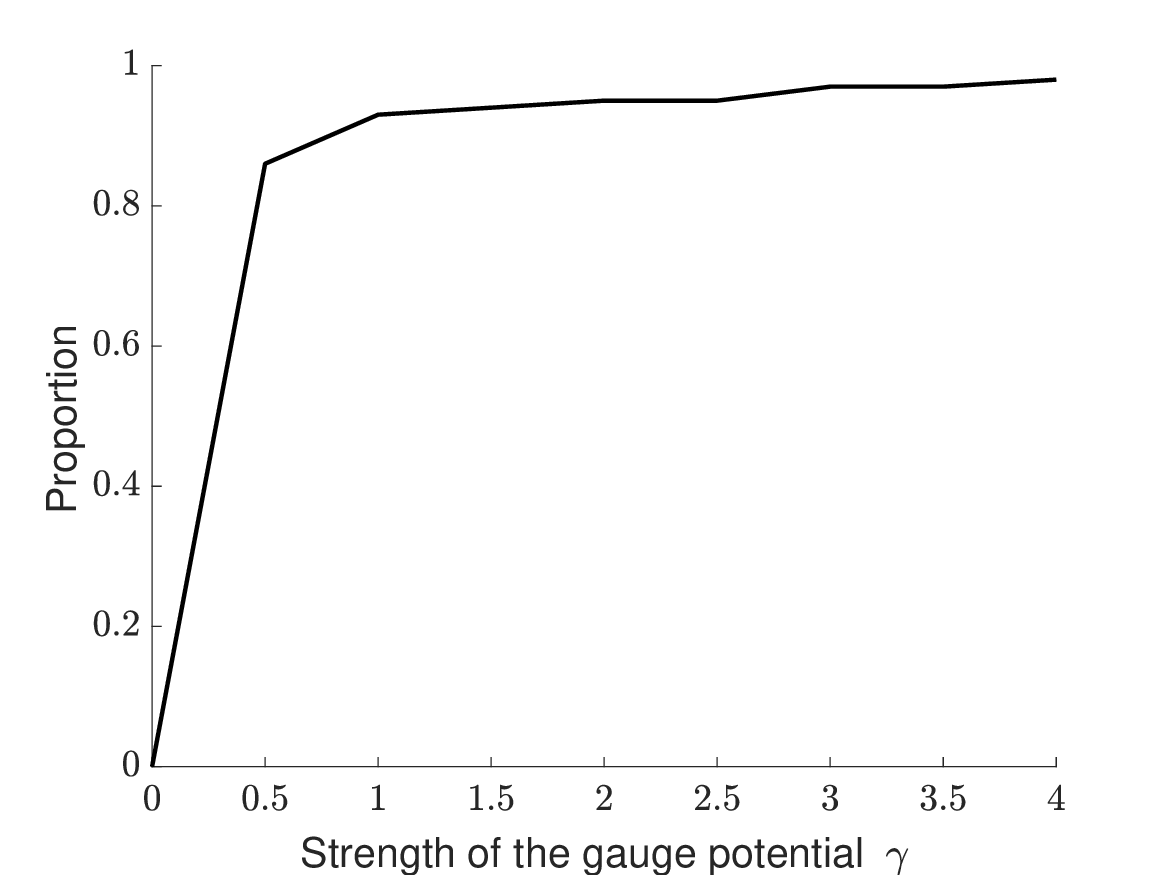}
    \caption{Proportion of condensated eigenmodes with $N=100$ and increasing $\gamma$.}
    \end{subfigure}
    \caption{Proportions of condensated eigenmodes.}
    \label{fig:condensated}
\end{figure}

To further quantify the non-Hermitian skin effect, we introduce the concept of degree of condensation.
 \begin{definition}
 Consider a chain of $N$ spherical resonators along the $x_1$-direction with centers at distinct $x_1$-coordinates  
 $\{{x_1}^i\}_{i=1,\ldots,n}$. For every eigenmode $v$, we define the degree of condensation as the vector $d_v\in\mathbb{R}^{N}$, where
    \begin{equation}
        (d_v)_i = \frac{\lVert{v|_{\{x\leq x_1^i\}}}\rVert_2}{\lVert{v}\rVert_2} \text{\ \ for $i=1,\ldots,N$}.
        \end{equation}
Here, the vector $v|_{\{x\leq x_1^i\}}$ has the same  first $i$ entries as $v$ but all the others are set to $0$. 
\end{definition}

    We now consider in Figure \ref{stability} the stability of the non-Hermitian skin effect in terms of the  positions of the resonators. Define the uniformly distributed random variables $\varepsilon_i \sim \mathcal{U}_{[-\varepsilon,\varepsilon]}$ for $i=1,\ldots,N$. We perturb the center of the resonators $x_1^i$ to $x_1^i(1+\varepsilon_i)$ and repeat the experiment $100$ times while fixing $\gamma = 1$. For each perturbation, we compute the average degrees of condensation. The stability result is similar to the one in the one-dimensional case \cite{ammari2023mathematical}.
\begin{figure}[H]
    \centering
    \begin{subfigure}[h]{0.48\textwidth}
    \includegraphics[width=\textwidth]{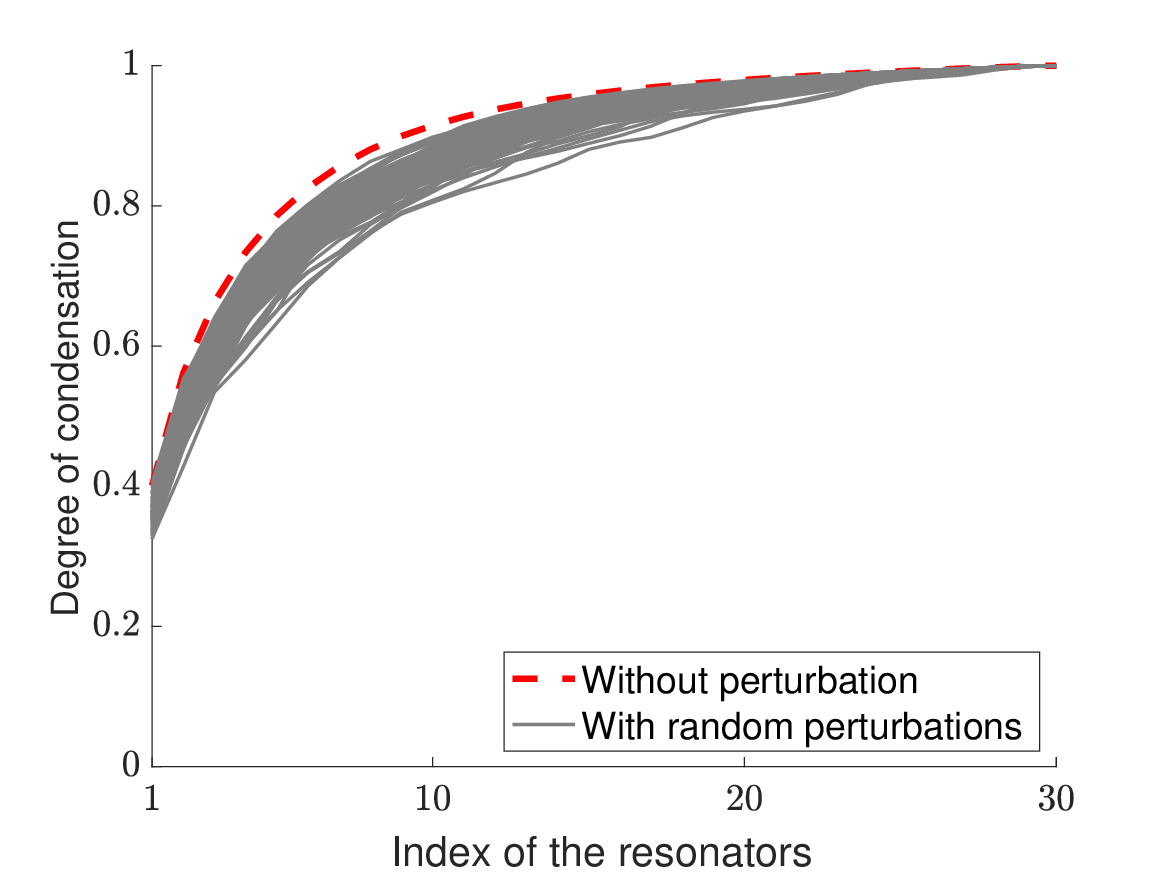}
    \caption{Average degrees of condensation with $\epsilon=0.1$, computed over $100$ realisations.}
    \end{subfigure}
    \begin{subfigure}[h]{0.48\textwidth}
    \includegraphics[width=\textwidth]{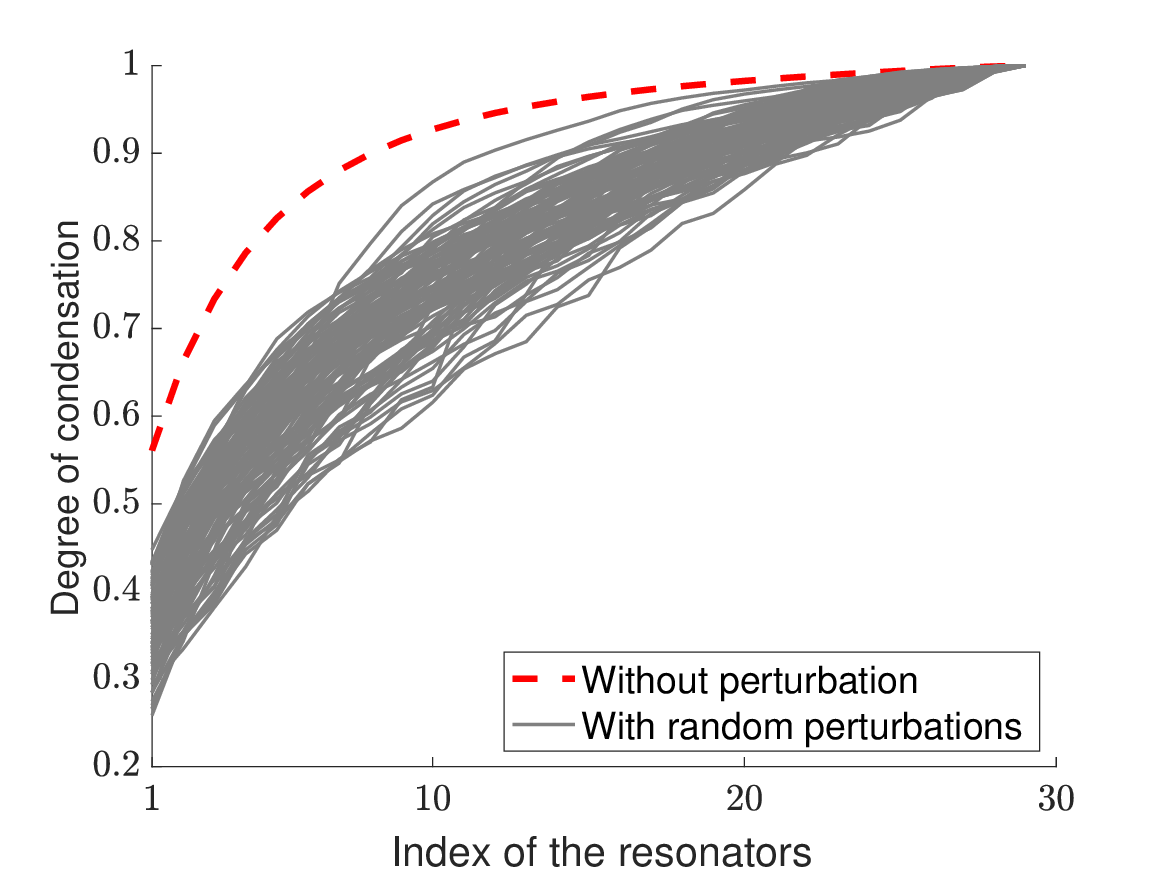}
    \caption{Average degrees of condensation with $\epsilon=0.2$, computed over $100$ realisations.}
    \end{subfigure}
    \caption{Average degrees of condensation for different strengths of geometric disorder. Grey lines are the randomly perturbed modes while the red line represents the unperturbed case. \label{stability}}
\end{figure} 

    Next, we consider in Figure \ref{stability2} the stability of the non-Hermitian skin effect in terms of $\gamma$. Define the uniformly distributed random variables $\varepsilon_i \sim \mathcal{U}_{[-\varepsilon,\varepsilon]}$ for $i=1,\ldots,N$. We perturb $\gamma$ to $\gamma (1+\varepsilon_i)$ in the $i$th resonator and compute the average degrees of condensation over $100$ runs while fixing the equidistant resonator structure. Again, the stability result is similar to the one in the one-dimensional case.
\begin{figure}[H]
    \centering
    \begin{subfigure}[h]{0.48\textwidth}
    \includegraphics[width=\textwidth]{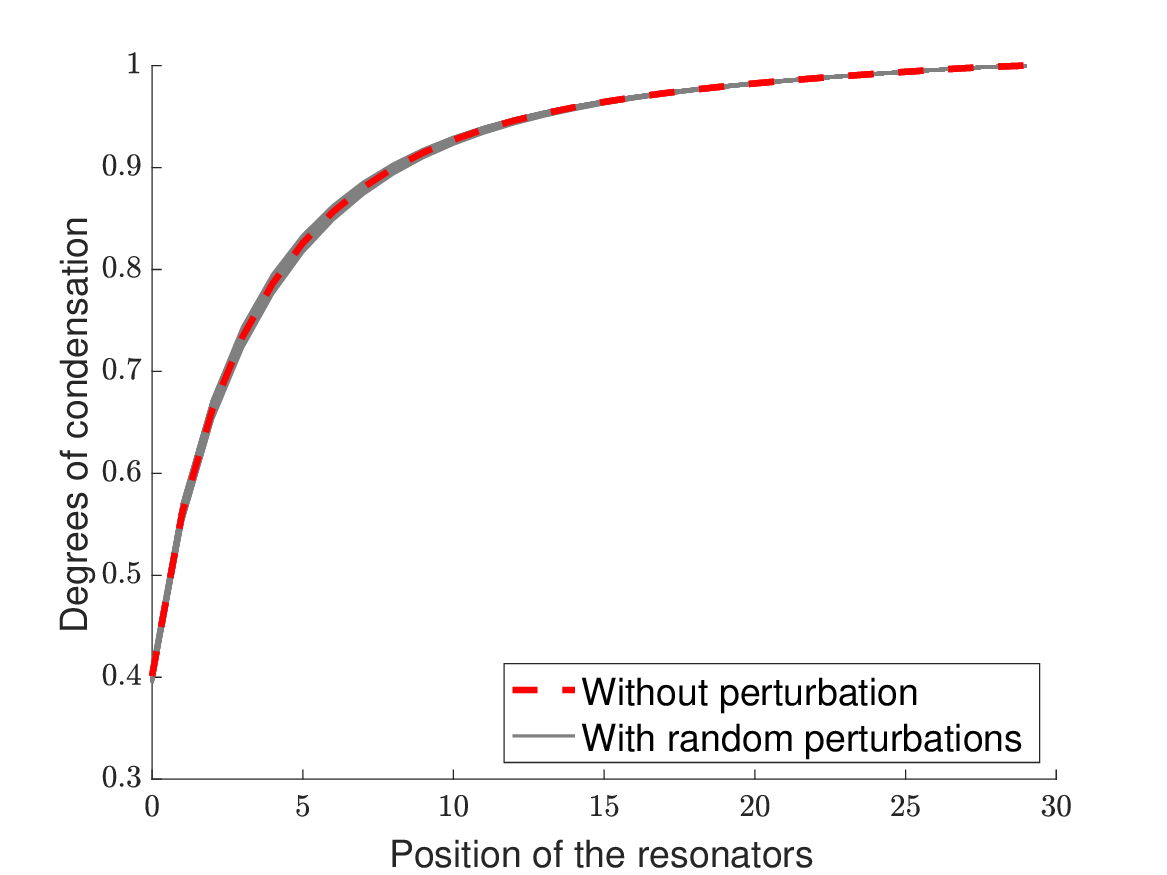}
    \caption{Average degree of condensation with $\epsilon=0.1$, computed over $100$ random realisations.}
    \end{subfigure}
    \begin{subfigure}[h]{0.48\textwidth}
    \includegraphics[width=\textwidth]{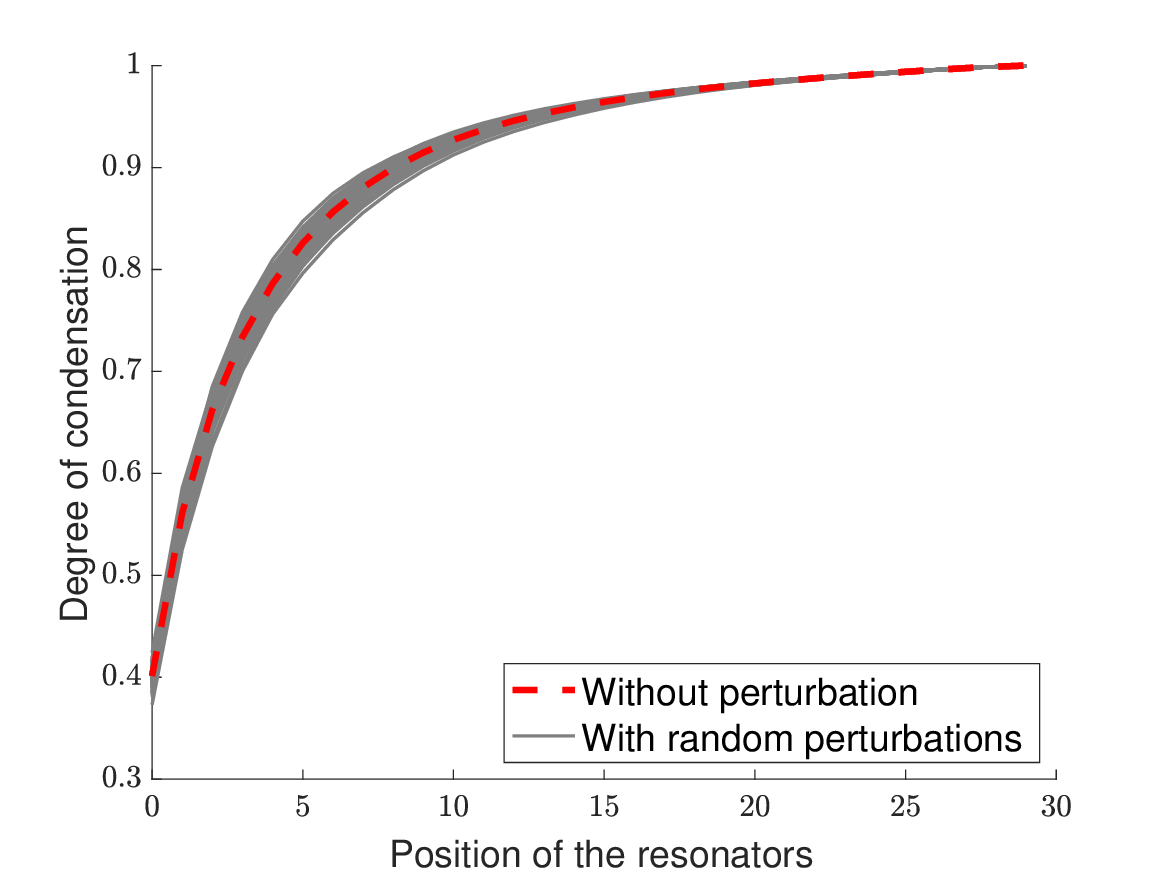}
    \caption{Average degree of condensation with $\epsilon=0.2$, computed over $100$ realisations.}
    \end{subfigure}
    \caption{Average degree of condensation for different strengths of disorder in $\gamma$. Grey lines are the randomly perturbed modes while the red line represents the unperturbed case. \label{stability2}}
\end{figure} 

\subsection{Numerical simulations of the non-Hermitian skin effect in other three-dimensional structures}
Until now, we have only considered the non-Hermitian skin effect in  chains of subwavelength resonators (i.e., structures with a one-dimensional lattice). In this section, we numerically demonstrate that the non-Hermitian skin effect also occurs in structures with higher dimension of the lattice. Without loss of generality, we can always, after translations and rotations, assume that the factor $\gamma$ aligns with the $x_1$-axis. Hence, the mathematical model \eqref{helmholtz} still holds. 

    In \Cref{fig:rect}, we consider a rectangle structure with two chains of resonators. We observe that $90\%$ of the eigenmodes are localised in the first $10$ resonators.
    \begin{figure}[!h]
    \centering
    \begin{subfigure}[h]{0.45\textwidth}
        \includegraphics[width=\textwidth]{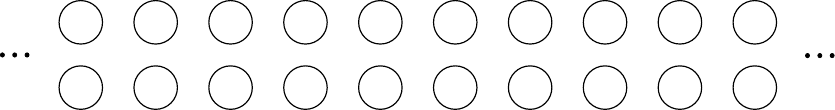}
        \caption{A rectangle structure with two lines and $100$ resonators in each line.}
        \label{fig:rectangle}
    \end{subfigure}
    \begin{subfigure}[h]{0.45\textwidth}
        \includegraphics[width=\textwidth]{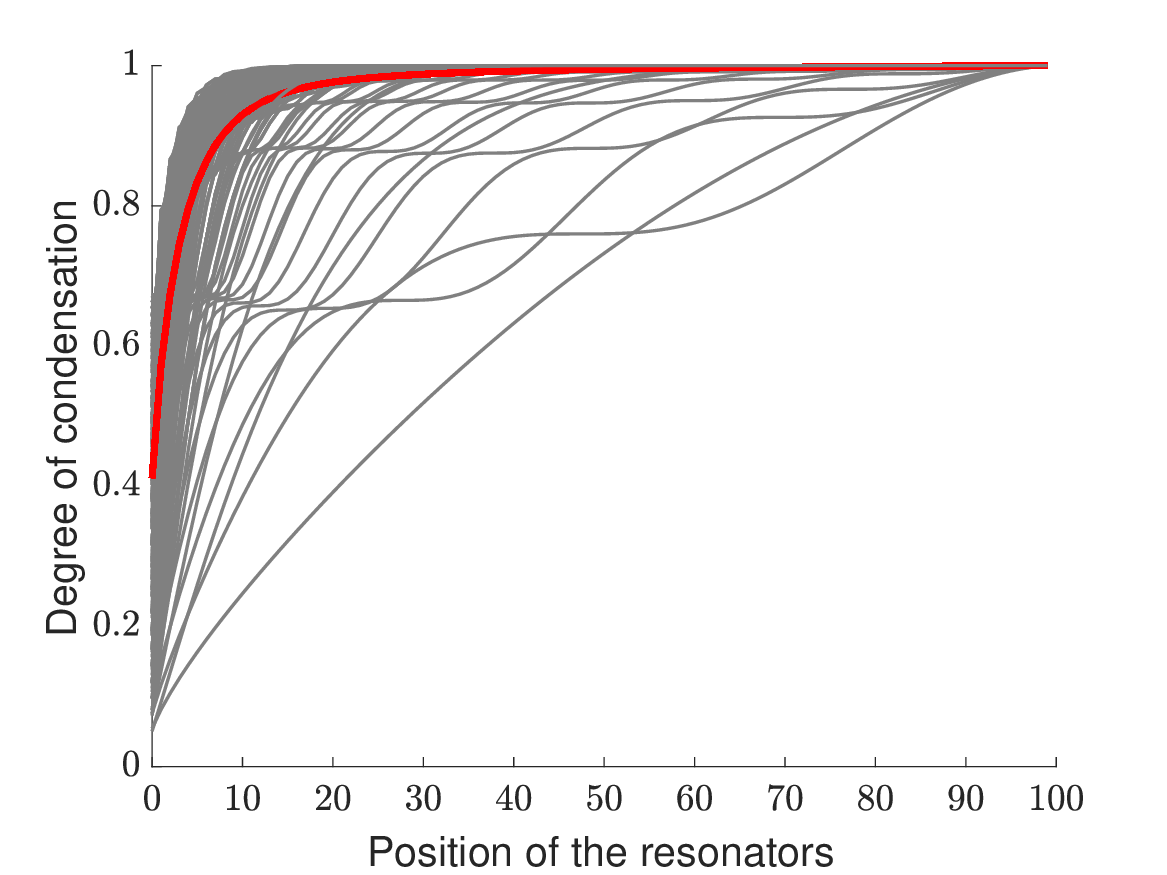}
        \caption{Degrees of condensation for all the eigenmodes of the rectangle for $\gamma =1$. The red line represents the average amplitude.}
        \label{fig:deg_rect}
    \end{subfigure}    
    \caption{Numerical simulations of the non-Hermitian skin effect  in a double chain system of resonators.}\label{fig:rect}
    \end{figure}

    Finally, in \Cref{fig:rhombus} we consider a rhombus structure with multiple lines. We observe that $90\%$ of the eigenmodes are localised on the left $20\%$ of the structure. The sudden jump of the degrees of condensation is due to the edge effects in each line.
 \begin{figure}[!h]
    \centering
    \begin{subfigure}[h]{0.45\textwidth}
        \includegraphics[width=\textwidth]{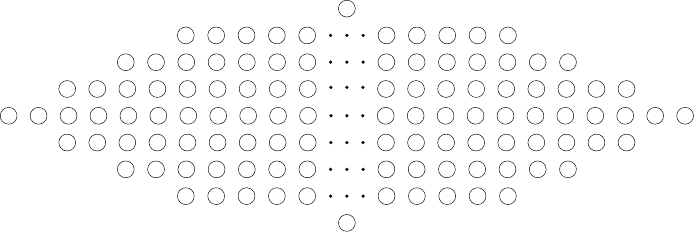}
        \caption{A rhombus structure with nine lines.}
        \label{fig:rectangle2}
    \end{subfigure}
    \begin{subfigure}[!h]{0.45\textwidth}
        \includegraphics[width=\textwidth]{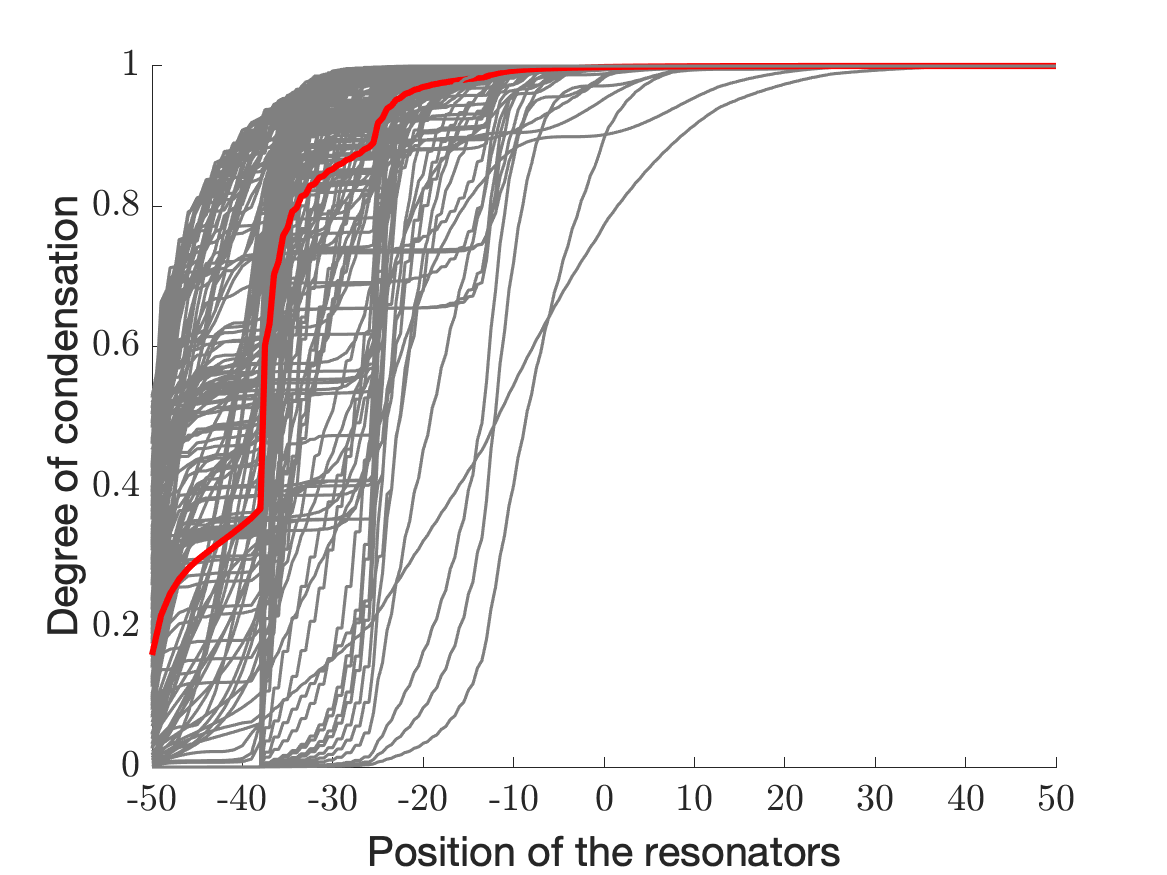}
        \caption{Degrees of condensation for all the eigenmodes of the rectangle structure for $\gamma =2$. The red line represents the average amplitude.}
        \label{fig:deg_rect2}
    \end{subfigure}
    \caption{Numerical simulations of the non-Hermitian skin effect  in  a rhombus of resonators.} \label{fig:rhombus}
    \end{figure}

\section{Concluding remarks}
In this paper, we have first introduced a discrete formulation for computing the subwavelength eigenfrequencies and eigenmodes of a system of subwavelength resonators with an imaginary gauge potential supported inside the resonators. This approximation is based on the so-called gauge capacitance matrix $\mathcal C_{N}^{\gamma}$. Unlike the one-dimensional case, due to long-range interactions in the system, the gauge capacitance matrix is dense. We have considered a range-$k$ approximation to keep the long-range interactions to a certain extent, thus obtaining a $k$-banded gauge capacitance matrix $\mathcal C_{N,k}^{\gamma}$. By proving exponential decay of the pseudo-vectors of the matrix through Toeplitz matrix theory, we have deduced the condensation of the eigenmodes at one edge of the structure. We have illustrated this non-Hermitian skin effect in a variety of examples and illustrated its stability with respect to imperfections in the system. Our results give mathematical foundations of the skin effect in non-Hermitian systems with long-range coupling in three dimensions.

A challenging further problem is to generalise these results to 
 dimer structures as we did in the one-dimensional case in \cite{dimerSkin}. Another challenging problem is to prove that when the strength of the disorder increases, there is a competition between the non-Hermitian skin effect and Anderson localisation in the bulk as
 recently shown numerically in the one-dimensional case in \cite{SkinStability}. In connection with this, it would be important to prove that all the eigenvalues of the gauge capacitance matrix are real (as illustrated in Figures \ref{fig:symbol1} and \ref{fig:symbolm1}) and random perturbations of the positions of the resonators or/and the parameter $\gamma$ inside the resonators leave them on the real axis but make some of them jump outside the region with non-trivial winding numbers.\\
 
\vspace{1cm}
\noindent
\textbf{Data Availability}

The data that support the findings of this study are openly available at \url{https://github.com/jinghaocao/skin_effect}.

\vspace{0.5cm}
\noindent
\textbf{Acknowledgments}

 This work was supported by Swiss National Science Foundation grant number 200021--200307 and by the Engineering and Physical Sciences
Research Council (EPSRC) under grant number EP/X027422/1.

\appendix
\section{Spectra and pseudo-spectra of Toeplitz matrices and operators}
In this section, we first recall some basic results on the spectra and pseudo-spectra of  Toeplitz matrices, Toeplitz operators, and Laurent operators. Then we characterise the pseudo-spectra of  banded Toeplitz matrices with perturbations. This characterisation is useful for studying the gauge capacitance matrix defined in \eqref{capacitancedef}.

\subsection{Spectra and pseudo-spectra of Toeplitz matrices}
An $N \times N$ Toeplitz matrix is a matrix whose entries are constant along diagonals:
\begin{equation} \label{eqA}
\mathbf{A}_N=\left(\begin{array}{ccccc}
a_0 & a_{1} & & \cdots & a_{N-1} \\
a_{-1} & a_0 & & & \vdots \\
& \ddots & \ddots & \ddots & \\
\vdots & & & a_0 & a_{1} \\
a_{-(N-1)} & \cdots & & a_{-1} & a_0
\end{array}\right) .
\end{equation}
A semi-infinite matrix of the same form is known as a Toeplitz operator, and a doubly infinite matrix of this kind is a Laurent operator and we denote them by $\vect A$. The symbol of a Toeplitz matrix, Toeplitz operator, or Laurent operator is the function
$$
f(z)=\sum_{k} a_k z^k
$$
with $\sum_k \babs{a_k}< +\infty$. This symbol is important for deriving many properties of the Toeplitz matrices, Toeplitz operators, and Laurent operators. For example, it fully determines the spectrum $\Lambda$ of a Toeplitz or a Laurent operator $\vect A$. Define $\mathbb T:= \{z\in \mathbb C: \babs{z}=1\}$ and $I(f(\mathbb T), \lambda)$ as the winding number of $f(\mathbb T)$ about $\lambda$ in the usual positive (counterclockwise) sense. We have the following result. 
\begin{thm} \label{thmapp1}
Let $\vect A$ and $f$ be as described above. The following holds:
\begin{enumerate}[(i)]
\item If $\vect A$ is a Laurent operator, then $\Lambda(\vect A)=f(\mathbb T)$;
\item If $\vect A$ is a Toeplitz operator, then 
$\Lambda(\vect A)=f(\mathbb T) \cup\{\lambda \in \mathbb{C}: I(f(\mathbb T), \lambda) \neq 0\}$.
\end{enumerate}
\end{thm}

Note that the winding number $I(f(\mathbb T), \lambda)$ or index of the continuous curve $f(\mathbb T)$ about the point $\lambda \in \mathbb{C}$ is also given by
$$
I(f(\mathbb T), \lambda)=\frac{1}{2 \pi i} \int_{\mathbb T} \frac{f^{\prime}(z)}{f(z)-\lambda} d z, \quad \lambda \notin f(\mathbb T) .
$$
Consequently, if $\lambda \in f(\mathbb T)$, then $I(f(\mathbb T), \lambda)$ is undefined.

A specific property of the Toeplitz operator $\vect A$ is that, for the eigenvectors of $\vect A$ corresponding to the eigenvalue $\lambda$ with $I(f(\mathbb T), \lambda) \neq 0$, the amplitude of its entries $a_j$'s decreases or increases as $j\rightarrow +\infty$. As stated in \cite[Eq. (3.4)]{reichel1992eigenvalues}:
\[
\begin{aligned}
\text{If } I(f(S), \lambda)>0, &\text{ then the right eigenvectors of $\vect A$ are geometrically decreasing};\\ 
\text{If } I(f(S), \lambda)<0, &\text{ then the right eigenvectors of $\vect A$ are geometrically increasing}. 
\end{aligned}
\]
This is known to be the topological origin of the skin effect in non-Hermitian physical systems with complex gauge transforms \cite{ammari2023mathematical}. 

For characterising the spectrum of Toeplitz matrices, no results similar to those in Theorem \ref{thmapp1} are available in the literature. Moreover, a detailed characterisation of general Toeplitz matrices seems to be out of reach. The only known results are asymptotic characterisations of eigenpairs of some Hermitian Toeplitz matrices and Toeplitz matrices with specific symbols; see \cite{bottcher2017asymptotics} for a survey in this regard. We also refer the reader to some representations of eigenpairs of tridiagonal Toeplitz matrices \cite{da2007characteristic, gover1994eigenproblem, yueh.cheng2008Explicit}, which led to our recent demonstration of the skin effect for finitely periodic systems of subwavelength resonators with complex gauge transformations in the one-dimensional case \cite{ammari2023mathematical, dimerSkin}. 

On the other hand, there are some results for the pseudo-spectra of banded and semi-banded Toeplitz matrices. In order to state them, we introduce the concept of pseudo-eigenvalue. 
\begin{definition}
Given $\varepsilon>0$, the number $\lambda \in \mathbb{C}$ is a pseudo-eigenvalue of a matrix $\vect A$ if any of the following equivalent conditions is satisfied:
\begin{enumerate}[(i)]
    \item $\lambda$ is an eigenvalue of $\vect A+\vect E$ for some $\vect E \in \mathbb{C}^{N \times N}$ with $\bnorm{\vect E} \leqslant \varepsilon$;
    \item $\exists \ u \in \mathbb{C}^N,\bnorm{u}=1$, such that $\bnorm{(\vect A-\lambda I) u} \leqslant \varepsilon$;
    \item $\bnorm{(\lambda I-\vect A)^{-1}} \geqslant \varepsilon^{-1}$;
    \item $\sigma_{\min}(\lambda I-\vect A) \leqslant \varepsilon$, where $\sigma_{\min}(\lambda I-\vect A)$ is the smallest singular value of 
   $\lambda I-\vect A$.
\end{enumerate}
\end{definition}

Given  $\varepsilon>0$, the set of all pseudo-eigenvalues of $\vect A$, the pseudo-spectrum, is denoted by $\Lambda_{\varepsilon}(\vect A)$ or simply $\Lambda_{\varepsilon}$.

An $l$-banded Toeplitz matrix is defined as $\vect A_{N}$ in \eqref{eqA} with $a_j=0$ for $\babs{j} > l$ and is denoted by $\vect A_{N,l}$. To state the result on its pseudo-spectrum, we need to introduce the sets $\Omega_r$ and $\Omega^R$ defined by
\begin{equation}\label{equ:defineofomegar}
{\Omega}_r=\left\{z \in \mathbb{C}: I\left(f\left(\mathbb T_r\right), z\right)>0\right\}, \quad {\Omega}^R=\left\{z \in \mathbb{C}: I\left(f\left(\mathbb T_R\right), z\right)<0\right\},
\end{equation}
where $\mathbb T_r:=\{z\in \mathbb C: \babs{z}=r\}$. The following theorem (\cite[Theorem 3.2]{reichel1992eigenvalues}) characterises the pseudo-eigenvalues of an $l$-banded Toeplitz matrix.  

\begin{thm}\label{thm:peudovectorbandedToeplitz1}
Let $\vect A$ be a banded Toeplitz operator with bandwidth $l$, i.e., $a_k=0$ for $|k|>l$. Let $f(z)$ be the symbol of $\vect A$, and let $\vect A_N$ denote the $N \times N$ Toeplitz matrix defined by (\ref{eqA}). Then, for any $r<1$ and $\rho>r$, we have
$$
\Omega_r \cup \Omega^{1 / r} \cup\left(\Lambda(\vect A)+\Delta_{\varepsilon}\right) \subseteq \Lambda_{\varepsilon}\left(\vect A_{N,l}\right) \quad \text { for } \quad \varepsilon=C \rho^N,
$$
where $\Delta_{\epsilon}=\{z\in \mathbb C: \babs{z}<\epsilon\}$ and $C$ is a constant independent of $N$. 
\end{thm}

\subsection{Pseudo-spectra of banded Toeplitz matrices with perturbations}
Note that our gauge capacitance matrix is a Toeplitz matrix with perturbations on the non-zero elements. Here, we consider and analyse the pseudo-spectrum of an $l$-banded Toeplitz matrix with perturbations on the corner blocks. More precisely, we define the $l$-banded Toeplitz matrix with the first and last $l$ rows perturbed as follows:
\begin{equation}\label{equ:perturbedtoeplitzmatrix1}
\widehat {\mathbf A}_{N, l}= \begin{pmatrix}
B_1\\
B_2\\
B_3
\end{pmatrix},
\end{equation}
where $B_1, B_2,$ and $B_3$ are respectively of size $(l-1)\times N$,  $(N-2l+2)\times N$, and  $(l-1)\times N$ and are given by 
\[
\begin{pmatrix}
a_0+b_{0,1} & \cdots & a_{l-1}+b_{l-1,1}& 0& \cdots &  \cdots &  \cdots &  \cdots  \\
a_{-1}+b_{-1,2} & \cdots & a_{l-2}+b_{l-2,2} & a_{l-1}+b_{l-1,2}& 0& \ddots &\ddots &\vdots \\
\vdots  & \ddots & \ddots & \ddots & \ddots &\ddots&\ddots &\vdots\\
a_{2-l}+b_{2-l,l-1} & \cdots & a_0+b_{0,l-1} & a_{1}+b_{1,l-1}&\cdots &  a_{l-1}+b_{l-1,l-1} &0 &\cdots
\end{pmatrix},
\]
\[
\begin{pmatrix}
a_{-l+1} & a_{-l+2} & \cdots & a_{l-1}& 0& \cdots &0  \\
 0 &a_{-l+1} & a_{-l+2} & \cdots & a_{l-1}& 0 &\vdots \\
 \vdots & \ddots & \ddots & \ddots & \ddots &\ddots&\vdots \\
 0& \cdots&0 &a_{-l+1} &a_{-l+2} & \cdots  &  a_{l-1}
\end{pmatrix},
\]
and 
\[
\begin{pmatrix}
0&\cdots &0&a_{1-l}+b_{1-l, N-l+2}  & \cdots & \cdots &\cdots & a_{l-3}+b_{l-3, N-l+2} & a_{l-2}+b_{l-2, N-l+2} \\
0&\cdots &0 & 0 &\ddots & \ddots & \ddots& \cdots &a_{l-3}++b_{l-3, N-l+3} \\
\vdots&\ddots &\vdots & \ddots & \ddots & \ddots & \ddots & \ddots  &\vdots\\
0&\cdots & 0 & 0& \cdots&0 &a_{1-l}+b_{1-l, N} & \cdots  &  a_{0}+b_{0, N}
\end{pmatrix}.
\]
We generalise Theorem \ref{thm:peudovectorbandedToeplitz1} to the perturbed $l$-banded Toeplitz $\widehat {\mathbf A}_{N,l}$ as follows. 
\begin{thm}\label{thm:peudovectorperturbToeplitz1}
Let $f(z)$ be the symbol defined by $\sum_{q=1-l}^{l-1}a_q z^{q}$, and let $\widehat{\mathbf A}_{N,l}$ be the $N \times N$ $l$-banded perturbed Toeplitz matrix in (\ref{equ:perturbedtoeplitzmatrix1}). Then, for any $r<1$ and $\rho>r$, we have
\begin{equation}\label{equ:peudovectorperturbToeplitz1}
\Omega_r \cup \Omega^{1 / r} \subseteq \Lambda_{\varepsilon}\left(\widehat {\mathbf A}_{N,l}\right) \quad \text { for } \quad \varepsilon=\max\left(C_1, C_2 N^{l-1}\right) \rho^{N-2l+2},
\end{equation}
where $\Omega_r, \Omega^{\frac{1}{r}}$ are defined as in (\ref{equ:defineofomegar}) and $C_1, C_2$ are constants independent of $N$. Moreover, there exist non-zero pseudo-eigenvectors $\mathbf{v}^{(N)}$ satisfying
\begin{equation}\label{equ:pseudoeigenvectorequ-2}
\frac{\bnorm{\left(\widehat {\mathbf{A}}_{N,l}-\lambda I\right) \mathbf{v}^{(N)}}_2}{\bnorm{\mathbf{v}^{(N)}}_2} \leq \max\left(C_1, C_2 N^{l-1}\right) \rho^{N-2l+2},
\end{equation}
such that
\begin{equation}\label{equ:pseudoeigenvectorequ-1}
\frac{\left|\left(\vect v^{(N)}\right)_j\right|}{\max_{j}\babs{\left(\vect v^{(N)}\right)_j}} \leq\left\{\begin{array}{ll}
C_3\rho^{j-1}, & \text { if } \lambda \in \Omega_r, \\
C_3\rho^{N-j}, & \text { if } \lambda \in \Omega^{\frac{1}{r}},
\end{array} \quad 1 \leq j \leq N, \right.
\end{equation}
where $C_3$ is independent of $N$. 
\end{thm}
\begin{proof}
By symmetry, $\Omega^{1 / r}$ must satisfy an estimate of  type (\ref{equ:peudovectorperturbToeplitz1}) if $\Omega_r$ does. Thus, we only need to prove that $\Omega_r \subseteq \Lambda_{\varepsilon}\left(\widehat {\mathbf A}_{N,l}\right)$.

The idea is to construct geometrically decreasing pseudo-eigenvectors. Given any $r<1$, let $\lambda \in \Omega_r$ be arbitrary. Assume without loss of generality that $a_{-l+1} \neq 0$. The symbol is 
\[
f(z)= \sum_{q=-l+1}^{l-1} a_q z^{q}.
\]
Then $f(z)$ has a pole of order exactly $l-1$ at $z=0$. Since $I\left(f\left(\mathbb T_r\right), \lambda\right) \geqslant 1$ (by the definition of $\Omega_r$), it follows by the argument principle that the equation $f(z)=\lambda$ has at least $l$ solutions in $\{z\in \mathbb C: \babs{z}<r\}$, counted with their multiplicities. Let $z_1, \ldots, z_l$ be any $l$ of these solutions. Assume for the moment that the $z_j$'s are distinct. Corresponding to each $z_j$ is a vector $u_j=\left(1, z_j^{1}, z_j^{2}, \ldots, z_j^{(N-1)}\right)^T$. 
Since $\lambda = f(z_j) = \sum_{q=-l+1}^{l-1} a_q z_j^{q}$, we can compute that
\[
\left((\lambda I - \widehat {\mathbf A}_{N,l})u_j\right)_{q} = 0, \quad l\leq q \leq N-l,
\]
\[
\left((\lambda I - \widehat {\mathbf A}_{N,l})u_j\right)_{q} = \sum_{k=-l+1}^{2-q}a_kz_{j}^{k}z_j^{q-1}+  \sum_{k=1-q}^{l-1}-b_{k,q}z_{j}^{k}z_{j}^{q-1}, \quad 1\leq q\leq l-1,
\]
and 
\[
\left((\lambda I - \widehat {\mathbf A}_{N,l})u_j\right)_{q} = \sum_{k=N-q+1}^{l-1}a_kz_{j}^{k}z_j^{q-1} +  \sum_{k=1-l}^{N-q}-b_{k,q}z_{j}^{k}z_{j}^{q-1}, \quad N-l+2\leq q\leq N.
\]
Thus, we have
\begin{equation}\label{equ:proofpeudovectorperturbToeplitz1}
(\lambda I - \widehat {\mathbf A}_{N,l})u_j = z_j^{N-2l+2}Lu_j + z_j^{-N+2l-2}Ru_j,
\end{equation}
where $L$ is given by
\begin{align*} 
\left(\begin{array}{cccccccccccc}
0&\cdots&0&0& 0 &0&\cdots&0 & 0&\cdots&\cdots&0\\
0&\cdots&\vdots&\vdots & \vdots & \vdots &\vdots&\vdots& \vdots&\cdots&\cdots&0\\
0&\cdots&0&0& 0 &0&\cdots&0 & 0&\cdots&\cdots&0\\
-b_{1-l,\zeta}&\cdots&\cdots & \cdots &\cdots&-b_{l-2, \zeta}& a_{l-1}&\cdots&0 & \cdots &\cdots &0\\
0&\ddots&\ddots&\ddots & \ddots & \vdots&\vdots&\ddots&\ddots & \ddots&\ddots&0\\
0&\cdots&-b_{1-l,N-1}&\cdots&-b_{0, N-1}& -b_{1, N-1}&a_{2} & \cdots &a_{l-1} &0 &\cdots&0 \\
0&\cdots&\cdots&-b_{1-l,N}&\cdots&-b_{0, N}&a_{1} & \cdots &\cdots &a_{l-1}  &\cdots &0
\end{array}\right)
\end{align*}
with $\zeta = N-l+2$ and $R$ is given by
\begin{align*}
\left(\begin{array}{ccccccccccccc}
0&\cdots&a_{1-l}&\cdots & \cdots & a_{-1}& -b_{0,1}&-b_{1,1} &\cdots & -b_{l-1, 1}&\cdots&\cdots&0 \\
0&\cdots&0&a_{1-l} & \cdots & a_{-2} & -b_{-1,2} & -b_{0,2} & -b_{1,2} &\cdots & -b_{l-1, 2}&\cdots &0\\
0&\cdots&0&\ddots& \ddots & \vdots & \vdots&\ddots&\ddots&\ddots & \ddots&\ddots&0\\
0&\cdots&0&\ddots&\ddots & a_{1-l} &-b_{2-l, l-1}&\cdots&\cdots &\cdots&\cdots&\cdots&-b_{l-1,l-1}\\
0&\cdots&0&0&\cdots & 0 &0&\cdots&0 & 0&\cdots&\cdots&0\\
\vdots&\cdots&\vdots&\vdots&\vdots & \vdots & \vdots &\vdots&\vdots& \vdots&\cdots&\cdots&\vdots\\
0&\cdots&0&0&\cdots & 0 &0&\cdots&0 & 0&\cdots&\cdots&0\\
\end{array}\right).
\end{align*}
Since there are $l$ linearly independent vectors $R_{1:l-1}z_j^{-N+2l-2}u_j$ of length $l-1$, we can find $l$ complex numbers $c_j$ so that
\begin{equation}\label{equ:proofpeudovectorperturbToeplitz5}
\sum_{j=1}^{l}c_jR_{1:l-1}z_j^{-N+2l-2}u_j = \vect 0,
\end{equation}
and let
\[
\vect v=\sum_{j=1}^l c_j u_j.
\]
Then $z_j^{-N+2l-2} Ru_j =\vect 0$ in (\ref{equ:proofpeudovectorperturbToeplitz1}) and we have
\begin{equation}\label{equ:proofpeudovectorperturbToeplitz2}
\left(\lambda I-\widehat {\mathbf A}_{N,l}\right) \vect v=L \sum_{j=0}^l z_j^{N-2l+2} c_j u_j.
\end{equation}
We now need to relate the norm of the right-hand side of this equation to $\bnorm{\vect v}_2$. To do this, we write $\vect v=U c$, where $U$ is the $N \times l$ Vandermonde matrix whose columns are the vectors $u_j$, and $c=(c_1, \ldots, c_{l})^{\top}$. We let $D$ be the diagonal matrix with elements $z_1^{N-2l+2}, \ldots, z_l^{N-2l+2}$. Then, we have
\[
\bnorm{U D c}_2=\bnorm{U D U^{\dagger} U c}_2=\bnorm{U D U^{\dagger} \vect v}_2 \leqslant \kappa(U)\bnorm{D}_2\bnorm{\vect v}_2,
\]
where $U^{\dagger}$ denotes the pseudo-inverse of $U$ and $\kappa(U)=\sigma_1/\sigma_l$ is its condition number. Since $\bnorm{D}_2 \leqslant r^{N-2l+2}$, it follows that (\ref{equ:proofpeudovectorperturbToeplitz2}) implies
\begin{equation}\label{equ:proofpeudovectorperturbToeplitz3}
\frac{\bnorm{\left(\lambda I-\widehat {\mathbf A}_{N,l}\right) \vect v}_2}{\bnorm{\vect v}_2} \leqslant r^{N-2l+2} \kappa(U)\bnorm{L}_2.
\end{equation}
This completes the proof under the assumption that the roots $z_1, \ldots, z_l$ are distinct. We can handle the case of multiple roots as follows. If some of the roots $z_0, \ldots, z_l$ are confluent at some points $\lambda \in \bar{\Omega}_r$, we then analyse the problem by confluent Vandermonde matrix. More precisely, assuming that the root $z_j$ has multiplicity $k$, we have 
\[
f(z) -\lambda = P(z)(z-z_j)^k,
\]
which yields
\begin{equation}\label{equ:proofpeudovectorperturbToeplitz4}
f^{(1)}(z_j)=0,\ f^{(2)}(z_j) = 0,\ \ldots,\ f^{(k-1)}(z_j)=0.
\end{equation}
Now, we consider $g(z) = f(z)z^{s}, s=l-1,l,l+1, \ldots$ By (\ref{equ:proofpeudovectorperturbToeplitz4}), we have 
\begin{align*}
g^{(1)}(z_j) =& f^{(1)}(z_j)z_j^{s}+f(z_j)sz_j^{s-1}= f(z_j)sz_j^{s-1} = \lambda sz_j^{s-1}\\
g^{(2)}(z_j) =& \cdots = \lambda s(s-1)z_j^{s-2}\\
\vdots&\\
g^{(k-1)}(z_j) = & \cdots = \lambda s(s-1) \ldots (s-k+2)z_j^{s-(k-1)}.
\end{align*}
Since $f(z)z^s =\sum_{q=-l+1}^{l-1}a_q z^{q+s}$, for $s\geq l-1$, the above equalities give
\begin{align*}
\lambda sz_j^{s-1} = &\sum_{q=-l+1}^{l-1}a_q(q+s)z_j^{q+s-1}  \\
\lambda s(s-1)z_j^{s-2} = &\sum_{q=-l+1}^{l-1}a_q(q+s)(q+s-1)z_j^{q+s-2}  \\
&\vdots \\
\lambda s(s-1) \ldots (s-k+2)z_j^{s-(k-1)}=& \sum_{q=-l+1}^{l-1}a_q(q+s)(q+s-1)\ldots (q+s-k+2)z_j^{q+s-(k-1)}. 
\end{align*}
This implies that we can use 
\begin{align*}
\begin{pmatrix}
1\\
z_j\\
z_{j}^{2}\\
\vdots\\
\vdots\\
z_{j}^{(N-1)}
\end{pmatrix}, \quad \begin{pmatrix}
0\\
1\\
2z_j\\
3z_{j}^{2}\\
\vdots\\
\vdots\\
(N-1)z_{j}^{(N-2)}
\end{pmatrix},\quad \ldots,\quad  \begin{pmatrix}
0\\
0\\
\vdots\\
(k-1)!\\
k(k-1)\cdots 2z_j\\
\vdots\\
(N-1)\cdots(N-(k-1))z_{j}^{N-k-2}
\end{pmatrix},
\end{align*}
and other vectors corresponding to other roots $z_j$'s to construct the pseudo-eigenvector. Instead, we consider to use 
\begin{align}\label{equ:proofpeudovectorperturbToeplitz6}
\begin{pmatrix}
1\\
z_j\\
z_{j}^{2}\\
\vdots\\
\vdots\\
z_{j}^{(N-1)}
\end{pmatrix}, \quad \begin{pmatrix}
0\\
z_j\\
2z_j^2\\
3z_{j}^{3}\\
\vdots\\
\vdots\\
(N-1)z_{j}^{(N-1)}
\end{pmatrix},\quad \ldots,\quad  \begin{pmatrix}
0\\
0\\
\vdots\\
(k-1)!z_j^{k-1}\\
k(k-1)\cdots 2z_j^{k}\\
\vdots\\
(N-1)\cdots(N-(k-1))z_{j}^{N-1}
\end{pmatrix}
\end{align}
to construct the  pseudo-eigenvector.

The above discussion shows that by considering the confluent Vandermonde matrix, we can always find such $l$ linearly independent vectors like (\ref{equ:proofpeudovectorperturbToeplitz6}) to construct the pseudo-eigenvector. Then an analysis involving confluent Vandermonde matrices yields a bound analogous to (\ref{equ:proofpeudovectorperturbToeplitz3}) except with $r^{N-2l+2}$ replaced by an algebraically growing factor at worst as $N^l r^{N-2l+2}$. This proves (\ref{equ:pseudoeigenvectorequ-2}).

Now, we prove (\ref{equ:pseudoeigenvectorequ-1}). As the pseudo-eigenvector is constructed by $\vect v = \sum_{j=1}^l c_j u_j$ with $c_j$'s satisfying (\ref{equ:proofpeudovectorperturbToeplitz5}) and $u_j$'s from (\ref{equ:proofpeudovectorperturbToeplitz6}), by the form of vectors in (\ref{equ:proofpeudovectorperturbToeplitz6}), we obtain
\[
\frac{\left|\vect v_j\right|}{\max_{j}\babs{\vect v_j}} \leq
C_3\rho^{j-1}, \ 1 \leq j \leq N, \text { if } \lambda \in \Omega_r,
\]
with $C_3$ being independent of $N$. 


\end{proof}

\printbibliography
\end{document}